\documentclass[runningheads]{llncs}
\usepackage{graphicx}
\usepackage{wrapfig}
\usepackage{ucs}
\usepackage{dsfont}
\usepackage{mathrsfs}
\usepackage{mathtools}
\usepackage{amsmath}
\usepackage{amsfonts}
\usepackage{amssymb}
\usepackage{soul}
\usepackage[makeroom]{cancel}
\usepackage{bbm}
\usepackage[english]{babel}
\usepackage{ucs}
\usepackage{hyperref}
\usepackage{longtable}
\usepackage{wasysym}
\usepackage{verbatim}
\usepackage{multirow}
\usepackage{bussproofs}
\usepackage{setspace}
\usepackage{graphicx}
\usepackage{stmaryrd}
\usepackage{forest}
\forestset{smullyan tableaux/.style={for tree={math content},where n children=1{!1.before computing xy={l=\baselineskip},!1.no edge}{},closed/.style={label=below:$\times$},},}
\usepackage{longtable}
\usepackage[all]{xy}
\usepackage{wrapfig}

\usepackage{xcolor}
\usepackage[colorinlistoftodos,prependcaption,textsize=small]{todonotes}

\usepackage[most]{tcolorbox}
\usepackage{enumitem}

\usepackage{tikz}

\newcommand{\commentSabine}[1]{}

\newcommand{\blue}[1]{\textcolor{blue}{#1}}

\newcommand{\FDE}{\textsf{FDE}}

\newcommand{\BD}{\textsf{BD}}

\newcommand{\f}{\varphi}
\newcommand{\p}{\psi}

\newcommand{\Luk}{{\mathchoice{\mbox{\rm\L}}{\mbox{\rm\L}}{\mbox{\rm\scriptsize\L}}{\mbox{\rm\tiny\L}}}}

\newcommand{\weakrightarrow}{\rightarrowtriangle}

\newtheorem{convention}{Convention}

\usepackage{comment}

\begin{document}

\setlength{\jot}{0pt} 
\setlength{\abovedisplayskip}{2pt}
\setlength{\belowdisplayskip}{2pt}
\setlength{\abovedisplayshortskip}{1pt}
\setlength{\belowdisplayshortskip}{1pt}

\title{Constraint tableaux for two-dimensional fuzzy logics\thanks{The research of Marta B\'ilkov\'a was supported by RVO: 67985807. The research of Sabine Frittella and Daniil Kozhemiachenko was funded by the grant ANR JCJC 2019, project PRELAP (ANR-19-CE48-0006).}}
%
%
\author{
Marta B\'ilkov\'a\inst{1}\orcidID{0000-0002-3490-2083} \and
Sabine Frittella\inst{2}\orcidID{0000-0003-4736-8614}
\and
Daniil Kozhemiachenko\inst{2}\orcidID{0000-0002-1533-8034}}
\authorrunning{B\'ilkov\'a et al}
%
\institute{The Czech Academy of Sciences, Institute of Computer Science, Prague\\
\email{bilkova@cs.cas.cz}
\and
INSA Centre Val de Loire, Univ.\ Orl\'{e}ans, LIFO EA 4022, France\\
\email{sabine.frittella@insa-cvl.fr, daniil.kozhemiachenko@insa-cvl.fr}}
\maketitle              
\begin{abstract}
We introduce two-dimensional logics based on \L{}ukasiewicz  and G\"{o}del logics to formalize 
reasoning with graded, incomplete and inconsistent information.
The logics are interpreted on matrices, where the common underlying structure is the bi-lattice (twisted) product of the $[0,1]$ interval. The first (resp.\ second) coordinate encodes the positive (resp.\ negative) information one has about a statement. We propose constraint tableaux that provide a modular framework to address their completeness and complexity.

\end{abstract}
\section{Introduction}

\paragraph{A two-dimensional treatment of uncertainty.}
Belnap-Dunn four-valued logic $\BD$ \cite{Belnap19,dunn76,OmoriWansing2017}, also referred to as First Degree Entailment $\FDE$, provides a logical framework to reason with both incomplete and inconsistent information. In $\BD$, formulas are evaluated on the Belnap-Dunn square 
(Figure~\ref{fig:square:with:filter}, left) where the four 
values encode the information available about the formula:  $\{t,f,b,n\}$ (true, false, both, neither). Hence, $b$ and $n$ correspond to inconsistent and incomplete information respectively. 
The shift in perspective lies in 
the values encoding the information available about the formula, and not 
the intrinsic truth or falsity of the formula which may 
not be accessible.
This idea was generalized by introducing the algebraic notion of bilattices by Ginsberg \cite{ginsberg88} in the context of AI, and studied further in \cite{R2010PhD,JR2012}. Bilattices contain two lattice orders simultaneously: a truth order, and 
an information order. 
Belnap-Dunn square, the smallest interlaced bilattice, can be seen as the product bilattice of the two-element lattice 
where the four values are seen as pairs of classical values which can be naturally interpreted as representing two independent dimensions of information – the positive and the negative one. We can understand them as providing positive and negative support for statements independently. 

Non-standard probabilities \cite{dunn2010,KMR21} extend the idea of independent positive and negative support of a statement in presence of uncertainty. They 
quantify evidence for and evidence against 
(the positive and negative probabilistic information about) a statement $\f$ with a couple $p(\f)=(p^+(\f),p^-(\f)) \in [0,1]\times [0,1]$. The maps are such that $p^-(\f)=p^+(\neg\f)$, 
$p^+$ is a monotone map w.r.t.\ \BD\ entailment relation, and satisfies the import-export axiom $p^+ (\f \wedge \p) + p^+ (\f \vee \p) = p^+ (\f) + p^+ (\p)$.
Since formulas are interpreted in $\BD$, one cannot prove that $p(\neg\f)=1 - p(\f)$, and $p(\f)=p(\neg\f)=1$ can be the case when one has contradictory information about $\f$. 
The range of non-standard probabilities coincides with the carrier of 
the continuous extension of Belnap-Dunn square (Figure~\ref{fig:square:with:filter}, center), which we see as the pro\-duct bilattice of the unit real interval $[0,1]\odot [0,1]$ in \autoref{ssec:prelim}.\footnote{In the context of Nelson's paraconsistent logics such product construction has been called twisted product of algebras \cite{Vakarelov1977}, or twist structures~\cite[Chapter~8]{Odintsov2008}.} We employ expansions of this algebra in \autoref{ssec:def:logics} to provide semantics to two-dimensional fuzzy logics.\footnote{We wish to stress we do not claim that non-standard probabilities are compositional or propose an algebraic interpretation of them.}
\vspace*{-5pt}
\paragraph{A broader motivation.}
This paper is a part of the project introduced in \cite{BilkovaFrittellaMajerNazari2020} aiming to de\-ve\-lop a modular logical framework for reasoning based on uncertain, incomplete and inconsistent information.
We model agents who build their epistemic attitudes (like beliefs) based on 
information aggregated from multiple sources. 
A convenient framework to formalize such 
reasoning is that of two-layer modal logics, first introduced in \cite{faginHalpernMegido1990,Hajek1998} and further developed in \cite{CintulaN14,baldietal2020}. Roughly speaking, the lower layer of events or evidence encodes the information given by the sources, while the upper layer encodes reasoning with the agent's attitudes based on this information, and the modalities expressing the attitudes connect the two layers and are interpreted in terms of an uncertainty measure (like probability, belief function, etc.).
In this article, we study two families of logics suitable for the upper layer. 

\vspace*{-5pt}
\paragraph{The logics.} 
We aim at a two-dimensional formalism that separates the positive and negative dimensions of information or support not only on the level of evidence, but also on the level of reasoning with agent's epistemic attitudes. In \cite{BilkovaFrittellaMajerNazari2020}, we have proposed examples of such
two-layer modal logics of belief based on incomplete and inconsistent information.
%
In the two-layer framework, the upper logic operates atomic propositions of the form $B\phi$ where $\phi$ is a formula of the lower layer (and the belief $B$ modalities do not nest). 
Atomic propositions of the logics we propose here can therefore be given such an epistemic interpretation, depending on a choice of epistemic attitudes and the uncertainty measure used to quantify evidence for and evidence against a statement. The logics themselves then model graded reasoning with such
epistemic attitudes.\footnote{This is a natural point to enter the discussion  whether reasoning about uncertainty can be adequately handled within truth-functional semantics (see e.g. \cite{DD08}).
Such discussion is however beyond the scope of the current paper.}

%
%
In some scenarios, it is reaso\-na\-ble to represent agents attitudes as probabilities (e.g.\ a company reasoning with information based on statistical data). To model graded reasoning about such attitudes, we propose logics derived from \L{}ukasiewicz logic \cite[Chapter VI]{FuzzyHB2}, mainly
because its language allows to express the (non-standard) probability axioms which is crucial to obtain complete axiomatization of the resulting two-layer logics \cite{Hajek1998,BilkovaFrittellaMajerNazari2020}. 
%

In other cases, the agent's aggregated attitude is not a probability. For instance, agents may be able to compare their belief on two different statements while not being necessarily able to say exactly to what extent they believe. Just as \L{}ukasiewicz logic can be seen as a logic of measure or quantity, G\"{o}del logic \cite[Chapter VII]{FuzzyHB2} can be considered a logic of order. 
In this context it is therefore natural to consider G\"odel logic as the starting point.

To comply with the two-dimensionality aim, we  define the logics semantically, using expansions of the pro\-duct bilattice $[0,1]\odot [0,1]$ with connectives derived from standard semantics of \L{}ukasiewicz logic or G\"{o}del logic. Two-dimensional treatment of implication is of a particular interest (as we explain more in detail in Remarks~\ref{rem:coimplication} and \ref{rem:negandimpl}). We consider two possibilities: the first dualizes implication by co-implication, 
the second understands negative support of an implication as a conjunction of the positive support of the antecedent and the negative support of the consequent. The first option connects to one of Wansing's logic of \cite{Wansing2008}, namely $I_4C_4$, and goes back to bi-intuitionistic logic \cite{Go00,Raus80}, the second option connects to Nelson's logic $N4$ \cite{Nelson1949}.

Depending on the choice of connectives, and the choice of the set of designated values on the resulting algebra, we encounter both logics which are paraconsistent and logics which are not. 
%
%
Before proceeding further, we need to clarify the notion of paraconsistency. Unless specified otherwise, we construe ‘logic’ as a set of valid formulas, not as sets of valid entailments. Hence, while not all logics considered in the paper lack explosion w.r.t. their entailment --- $p,\neg p\vDash q$, in none of them  $(p\wedge\neg p)\rightarrow q$ is valid. It is in this sense that we call the logics presented here ‘paraconsistent’.
\vspace*{-5pt}
\paragraph{Proof theory.}
Proof theory for many-valued logics is mostly presented in either of the following three forms. Hilbert style axiomatic calculi (cf., e.g.~\cite{Hajek1998,MetcalfeOlivettiGabbay2008}); different versions of sequent and hypersequent calculi~\cite{Haehnle2001HBPL,MetcalfeOlivettiGabbay2008}; tableaux and decomposition calculi (cf., e.g.~\cite{Haehnle2001HBAR}
for \L{}ukasiewicz logic and~\cite{AvronKonikowska2001} for G\"{o}del logic).

Each of these proof formalisms has its own advantages: Hilbert calculi provide an explicit list of postulates which facilitates establishing the relations between different logics (e.g. whether one logic is an extension of another). The rules of (hyper)sequent calculi provide structural insights into the algebraic properties of the connectives of the given logic. 
On the other hand, tableaux and decomposition systems are easily automatisable and can be readily used to determine an upper bound on the complexity of the validity and satisfiability problems for the logic in question. 
Another advantage of the tableaux is that their semantical nature allows for a straightforward formalisation of different entailment relations defined on the same algebra. Since the logics we are going to introduce are hybrids  between $\FDE$ and \L{}ukasiewicz or G\"{o}del logic, we opt for combining the constraint tableaux framework with the
\FDE-tableaux by D’Agostino \cite{DAgostino1990}.
\vspace*{-5pt}
\paragraph{Structure of the paper.}
Section \ref{sec:logics} presents preliminaries on bilattices and matrices and introduces the logics for a two dimensional treatment of uncertainty and their properties (proofs are in the Appendix).  Section \ref{sec:tableaux} presents the constraint tableaux for these logics and discusses their soundness and completeness, and the complexity of the proof search.  Section \ref{sec:conclusion} presents further lines of research. 
\section{The logics for a two-dimensional treatment of uncertainty} 
\label{sec:logics}
\subsection{Preliminaries}
\label{ssec:prelim}
First, we describe the algebras we are going to use to interpret the logics. Their construction relays on the standard MV-algebra, and the standard G\"{o}del algebra, which provide the standard semantics of \L{}ukasiewicz and G\"{o}del logic respectively (we refer the reader to \cite[Chapters~VI,VII]{FuzzyHB2} for a basic exposure to G\"{o}del and \L{}ukasiewicz logics and their standard semantics). 
In what follows, $[0,1]$ denotes the real unit interval with its natural order, and $[0,1]^\mathsf{op}$ denotes the interval with the reversed order.  
\paragraph{The standard MV-algebra}  $[0,1]_{\Luk}=([0,1],0,\wedge,\vee,\&,\rightarrow_{\Luk})$ is defined as follows: for all $a,b\in[0,1]$ the standard operations are given by
\begin{align*}
a\wedge b &\coloneqq\min(a,b) &
a\& b &\coloneqq\max(0, a+b-1)\\
a\vee b &\coloneqq\max(a,b) &
a\rightarrow_{\Luk}b &\coloneqq \min(1,1-a+b)
\end{align*}
Moreover, we define the negation ${\sim_{\Luk}}a\coloneqq a\rightarrow_{\Luk}0$, the constant $1\coloneqq\sim_{\Luk} 0$, the truncated sum
$a\oplus b\coloneqq \sim_{\Luk} a \rightarrow_{\Luk} b$, and the truncated subtraction $a\ominus b\coloneqq a\&{\sim}b$.

The MV-algebra $[0,1]^\mathsf{op}_{\Luk}= ([0,1]^\mathsf{op},1, \vee, \wedge, \oplus, \ominus)$ arises turning the standard MV-algebra upside down, and is isomorphic to it. Here, we have ${\sim_{\Luk}}a \coloneqq 1 \ominus a$.
\paragraph{The standard G\"{o}del algebra} $[0,1]_{\mathsf{G}}=([0,1],0,\wedge,\vee,\rightarrow_{\mathsf{G}})$ is defined as follows: for all $a,b\in[0,1]$, the standard operations are given by $a\wedge b \coloneqq \min(a,b), a\vee b \coloneqq \max(a,b)$, and the implication is defined as follows. We at the same time spell out a definition of a co-implication we shall need later on: 
\begin{equation*}
 a\rightarrow_{G} b=
\begin{cases}
1, \ \text{if}\  a\leq b\\
b \ \ \text{else}
\end{cases}
\ \ \ \qquad
b\Yleft_{G} a=
\begin{cases}
0, \ \text{if}\  b\leq a\\
b \ \ \text{else}
\end{cases}
\end{equation*}
We define a negation ${\sim_{\mathsf{G}}}a\coloneqq a\rightarrow_{\mathsf{G}}0$, and $1\coloneqq \sim_{\mathsf{G}}0$. 

The algebra $[0,1]^\mathsf{op}_{\mathsf{G}}=([0,1]^\mathsf{op},1,\vee,\wedge,\Yleft_{\mathsf{G}})$ arises by dualizing the standard G\"{o}del algebra (in particular, similarly as $\rightarrow_{G}$ is the residuum of $\wedge$, $\Yleft_{\mathsf{G}}$ is the residuum of $\vee$). A negation can be defined on this algebra as $-_{\mathsf{G}} a \coloneqq 1\Yleft_{\mathsf{G}} a$.
\begin{remark}\label{rem:coimplication}
Observe that $\ominus$ and $\Yleft_G$ are dual to $\rightarrow_{\L}$ and $\rightarrow_G$ in the following sense.
\begin{align*}
a\leq b\oplus c&\text{ iff }a\ominus b\leq c&a\&b\leq c&\text{ iff }a\leq b\rightarrow_{\L}c\\
a\leq b\vee c&\text{ iff }a\Yleft_G b\leq c&a\wedge b\leq c&\text{ iff }a\leq b\rightarrow_{G}c
\end{align*}
As one can see, $\ominus$ and $\Yleft_G$ residuate disjunctions dually to how $\rightarrow_{\L}$ and $\rightarrow_G$ residuate conjuctions. Taking these dualities into account, we will call $\ominus$ and $\Yleft_G$ \emph{co-implications}.
\end{remark}
\paragraph{Product billatices} Given an arbitrary lattice
$\mathbf{L}=(L,\wedge_L,\vee_L)$, 
we can construct the \emph{product bilattice} $\mathbf{L} \odot \mathbf{L}=(L \times L , \wedge, \vee, \sqcap, \sqcup, \neg)$ \cite{Avron96,AvronArieli1996}. In what follows, we  essentially use the product bilattice $[0,1]\odot [0,1]$, constructed from the lattice $([0,1], \min, \max)$.
We  only consider the $\{\wedge, \vee,\neg\}$ reduct of this structure in this paper, and not to complicate notation denote it by $[0,1]\odot [0,1]$. It is defined as follows:
for all $(a_1, a_2), (b_1,b_2) \in [0,1] \times [0,1]$, 
\begin{align*}
(a_1, a_2)\leq (b_1,b_2) &\coloneqq a_1\leq b_1 \text{ and } b_2\leq a_2\\
 \neg (a_1,a_2) &\coloneqq (a_2,a_1) \\
(a_1,a_2) \wedge (b_1,b_2) &\coloneqq (\min (a_1,b_1), \max (a_2,b_2))
\\ 
(a_1,a_2) \vee (b_1,b_2) &\coloneqq (\max (a_1 ,b_1), \min (a_2,b_2)).
\end{align*}
We  use expansions of $[0,1]\odot [0,1]$ by implication connectives derived from the \L{}ukasiewicz or G\"{o}del implication described above. Their positive support coincides with those of $\L$ and $\mathsf{G}$ implications. For the negative support, we consider two options. The first one dualizes the implication by the co-implication, the second results in negating implication by the conjunction of the positive part of the antecedent and the negative part of the consequent. For \L{}ukasiewicz logics these result in:
$$
(a_1,a_2) \rightarrow (b_1,b_2) \coloneqq (a_1\rightarrow_{\Luk}b_1,b_2\ominus a_2)\ \ 
(a_1,a_2) \weakrightarrow (b_1,b_2) \coloneqq (a_1\rightarrow_{\Luk}b_1,a_1 \& b_2)
$$
For G\"{o}del logics we obtain:
$$
(a_1,a_2) \rightarrow (b_1,b_2) \coloneqq (a_1\rightarrow_{\mathsf{G}}b_1,b_2\Yleft_{\mathsf{G}} a_2)\ \ 
(a_1,a_2)\weakrightarrow(b_1,b_2)\!\coloneqq\!(a_1\rightarrow_{\mathsf{G}}b_1,a_1 \wedge b_2)$$
In the first option, the interpretation arises as the one on the product algebra $[0,1]_{\Luk}\times[0,1]^\mathsf{op}_{\Luk}$ or $[0,1]_{\mathsf{G}}\times[0,1]^\mathsf{op}_{\mathsf{G}}$. In the G\"{o}del case, it relates to how the implication is interpreted in Wansing's logic $I_4C_4$ \cite{Wansing2008}.
In the second option, $\weakrightarrow$ is not congruential, and a strong congruential implication can be defined as $(a\weakrightarrow b)\wedge(\neg b\weakrightarrow\neg a)$. The second option corresponds to how implication is interpreted in product residuated bilattices of \cite{JR2012}. In the G\"{o}del case, it relates to how the implication is interpreted in Nelson's logic N4 \cite{Nelson1949}.

\begin{itemize}[noitemsep,topsep=2pt]
    \item We  denote by $[0,1]_{\Luk}\odot[0,1]_{\Luk}(\rightarrow)$ and $[0,1]_{\Luk}\odot[0,1]_{\Luk}(\weakrightarrow)$ the correspon\-ding expansions of $[0,1]\odot [0,1]$ defined using the \L{}ukasiewicz connectives.
    \item We  denote by $[0,1]_{\mathsf{G}}\odot[0,1]_{\mathsf{G}}(\rightarrow)$ and $[0,1]_{\mathsf{G}}\odot[0,1]_{\mathsf{G}}(\weakrightarrow)$ the correspon\-ding expansions of $[0,1]\odot [0,1]$ defined using the G\"{o}del connectives.
\end{itemize}
\subsection{The logics}
\label{ssec:def:logics}
The logics considered in this paper are defined through matrix semantics \cite{Font16}. We consider logical matrices of the form $(\mathbf{A},D)$ where $\mathbf{A}$ is one of the four algebras described above, and $D\subseteq A$ is a set of designated values. 
As sets of designated values, we  use various lattice filters of the form $(x,y)^{\uparrow}\coloneqq\{(x',y') \mid x \leq x' \text{ and } y'\leq y \}$ (see Figure~\ref{fig:square:with:filter}, center). The motivation is the following: $x$ represents the threshold of having enough evidence to say there is reasonable evidence supporting the truth of the formula, while $y$ represents the threshold below which one considers not to have enough evidence to say that there is reasonable evidence supporting the falsity of the formula. Of particular interest are filters $(1,0)^{\uparrow}$ 
(the evidence fully supports the formula and does not contradicts it)
and $(1,1)^{\uparrow}$ (there is some evidence that fully supports the formula).
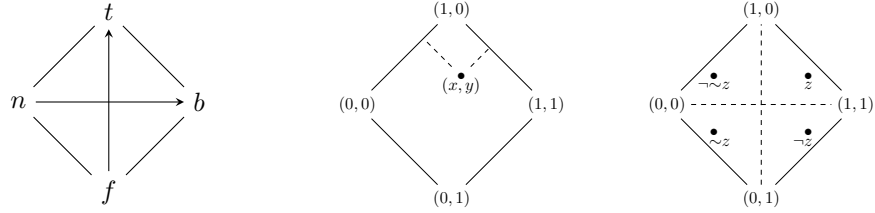
\begin{figure}[h]
\centering
  \begin{tikzpicture}[>=stealth,relative]
	\node (U1) at (0,-1.2) {$f$};
	\node (U2) at (-1.2,0) {$n$};
	\node (U3) at (1.2,0) {$b$};
	\node (U4) at (0,1.2) {$t$};
	
	\path[-,draw] (U1) to (U2);
	\path[-,draw] (U1) to (U3);
	\path[-,draw] (U2) to (U4);
	\path[-,draw] (U3) to (U4);
	\path[->,draw] (U1) to (U4);
	\path[->,draw] (U2) to (U3);
	\end{tikzpicture}
\qquad\qquad
\resizebox{0.27\textwidth}{!}{
\begin{tikzpicture}[>=stealth,relative]
\node (U1) at (0,-2) {$(0,1)$};
\node (U2) at (-2,0) {$(0,0)$};
\node (U3) at (2,0) {$(1,1)$};
\node (U4) at (0,2) {$(1,0) $};
\node (U5) at (0.2,0.6) {$\bullet$};
\node (U6) at (0.2,0.4) {$(x,y)$};
\path[-,draw] (U1) to (U2);
\path[-,draw] (U1) to (U3);
\path[-,draw] (U2) to (U4);
\path[-,draw] (U3) to (U4);
\draw[dashed] (U5) -- (0.8,1.2);
\draw[dashed] (U5) -- (-0.6,1.4);
\end{tikzpicture}
}
\quad\quad
\resizebox{0.27\textwidth}{!}{
\begin{tikzpicture}[>=stealth,relative]
\node (U1) at (0,-2) {$(0,1)$};
\node (U2) at (-2,0) {$(0,0)$};
\node (U3) at (2,0) {$(1,1)$};
\node (U4) at (0,2) {$(1,0) $};
\node (z) at (1,0.6) {$\bullet$};
\node (zl) at (1,0.4) {$z$};
\node (cz) at (-1,0.6) {$\bullet$};
\node (czl) at (-1,0.4) {$\neg{\sim}z$};
\node (nz) at (1,-0.6) {$\bullet$};
\node (zl) at (1,-0.8) {$\neg z$\ \ \ };
\node (lukz) at (-1,-0.6) {$\bullet$};
\node (lukzl) at (-1,-0.8) {\ \ ${\sim}z$};
\path[-,draw] (U1) to (U2);
\path[-,draw] (U1) to (U3);
\path[-,draw] (U2) to (U4);
\path[-,draw] (U3) to (U4);
\draw[dashed] (U1) -- (U4);
\draw[dashed] (U2) -- (U3);
\end{tikzpicture}
}
\caption{Belnap-Dunn square $\mathbf{4}$ (left), its continuous pro\-ba\-bilistic extension with the filter $(x,y)^\uparrow$ (center) and the geometric interpretation of $\neg$, ${\sim}$, and $\neg{\sim}$ for $\L^2$ (right).}
\label{fig:square:with:filter}
\end{figure}

Each logical matrix determines the set of valid formulas in a given language (formulas, which are, for each valuation, designated), and a consequence relation (an entailment) between sets of formulas and formulas, defined as  preservation of designated values.
Regarding the two types of implication introduced in the previous subsection, the standard semantics of the logics is set as follows: 
\begin{itemize}[noitemsep,topsep=2pt]
    \item Logics with the $\rightarrow$ implication in the language (i.e. the logics $\Luk^2_{(x,y)}(\rightarrow)$ and $\mathsf{G}^2_{(x,y)}(\rightarrow)$ below) are given by the matrices $([0,1]_{\Luk}\odot[0,1]_{\Luk}(\rightarrow),(x,y)^{\uparrow})$ and $([0,1]_{\mathsf{G}}\odot[0,1]_{\mathsf{G}}(\rightarrow),(x,y)^{\uparrow})$ respectively. 
    \item Logics with the $\weakrightarrow$ implication in the language (i.e. the logics  $\Luk^2_{(x,y)}(\weakrightarrow)$ and
$\mathsf{G}_{(x,y)}^2(\weakrightarrow)$ below) are given by the matrices $([0,1]_{\Luk}\odot[0,1]_{\Luk}(\weakrightarrow),(x,y)^{\uparrow})$ and $([0,1]_{\mathsf{G}}\odot[0,1]_{\mathsf{G}}(\weakrightarrow),(x,y)^{\uparrow})$ respectively. 
\end{itemize}
We  however need to treat \L{}ukasiewicz and G\"{o}del logics separately. Therefore it is practical to define the language and semantics for them separately in a compact way as follows. We refer by $\Luk^2$ to \L{}ukasiewicz logics, and by $\mathsf{G}^2$ to G\"{o}del logics, specifying the filter in the subscript.
\begin{definition}[Language and semantics of $\Luk^2$]
\label{def:semantics:L2}
We fix a countable set $\mathsf{Prop}$ of propositional letters and consider the following language:
\[\phi\coloneqq \mathbf{0}\mid p\mid\neg\phi\mid(\phi\wedge\phi)\mid(\phi\vee\phi)\mid(\phi\rightarrow\phi)\mid(\phi\weakrightarrow\phi)\]
where $p\in\mathsf{Prop}$. We define ${\sim}\phi\coloneqq\phi\rightarrow\mathbf{0}$, ${\sim}_w\phi\coloneqq\phi\weakrightarrow\mathbf{0}$, $\phi_1\odot\phi_2\coloneqq{\sim}(\phi_1\rightarrow{\sim}\phi_2)$, and $\phi_1\leftrightarrow\phi_2\coloneqq(\phi_1\rightarrow\phi_2)\odot(\phi_2\rightarrow\phi_1)$.

 Let $v:\mathsf{Prop}\rightarrow [0,1]\times [0,1]$, and denote $v_1$ and $v_2$ its left and right coordinates, respectively. We extend $v$ as follows.
\[\begin{array}{rclrcl}
v(\mathbf{0})&=&(0,1)&
v(\phi_1\wedge\phi_2)&=&(v_1(\phi_1)\wedge v_1(\phi_2),v_2(\phi_1)\vee v_2(\phi_2))
\\
v(\neg\phi)&=&(v_2(\phi),v_1(\phi)) \quad
&
v(\phi_1\vee\phi_2)&=&(v_1(\phi_1)\vee v_1(\phi_2),v_2(\phi_1)\wedge v_2(\phi_2))
\\
&&&
v(\phi_1\rightarrow\phi_2)&=&(v_1(\phi_1)\!\rightarrow_{\Luk}\!v_1(\phi_2),v_2(\phi_2)\ominus v_2(\phi_1))
\\
&&&
v(\phi_1\weakrightarrow\phi_2)&=&(v_1(\phi_1)\!\rightarrow_{\Luk}\!v_1(\phi_2),v_1(\phi_1)\ \& \ v_2(\phi_2))
\end{array}\]
Notice that
\[\begin{array}{rcl}
v({\sim}\phi)&=&(1-v_1(\phi),1-v_2(\phi)) \\
v(\phi_1\odot\phi_2)&=&(v_1(\phi_1)\ \& \ v_1(\phi_2),v_2(\phi_1)\oplus v_2(\phi_2))\\
v(\phi_1\leftrightarrow\phi_2)&=&(1-|v_1(\phi_1)-v_1(\phi_2)|,|v_2(\phi_1)-v_2(\phi_2)|)\\
\end{array}\]
\end{definition}
\begin{remark} \label{rk:symmetry}
In section \ref{ssec:sem:properties:L2}, we  use the fact that $\neg$ corresponds to a symmetry w.r.t.\ the horizontal axis, ${\sim}$  to a symmetry w.r.t.\ the point $(0.5,0.5)$, ${\sim}\neg$ and $\neg{\sim}$ are correspond  to a symmetry w.r.t.\ the horizontal axis (see Figure \ref{fig:square:with:filter}, right). From the meaning perspective, $\neg$ corresponds for swapping the positive and negative supports of the statement.
\end{remark}
\begin{definition}[Language and semantics of $\mathsf{G}^2$]
\label{def:semantics:G2}
We fix a countable set $\mathsf{Prop}$ of propositional letters and consider the following language:
\[\phi\coloneqq \mathbf{0}\mid\mathbf{1}\mid p\mid\neg\phi\mid(\phi\wedge\phi)\mid(\phi\vee\phi)\mid(\phi\rightarrow\phi)\mid(\phi\Yleft\phi)\mid(\phi\weakrightarrow\phi)\]
where $p\in\mathsf{Prop}$. We define ${\sim}\phi\coloneqq\phi\rightarrow\mathbf{0}$, and ${\sim}_w\phi\coloneqq\phi\weakrightarrow\mathbf{0}$.

Let $v:\mathsf{Prop}\rightarrow [0,1]\times [0,1]$, and denote $v_1$ and $v_2$ its left and right coordinates, respectively. We extend $v$ as follows.
\[\begin{array}{rclrcl}
v(\mathbf{0})&=&(0,1)&v(\phi_1\wedge\phi_2)&=&(v_1(\phi_1)\wedge v_1(\phi_2),v_2(\phi_1)\vee v_2(\phi_2))
\\
 v(\mathbf{1})&=&(1,0)
&v(\phi_1\vee\phi_2)&=&(v_1(\phi_1)\vee v_1(\phi_2),v_2(\phi_1)\wedge v_2(\phi_2))
\\
v(\neg\phi)&=&(v_2(\phi),v_1(\phi))
\qquad
&v(\phi_1\rightarrow\phi_2)&=&(v_1(\phi_1)\!\rightarrow_{\mathsf{G}}\!v_1(\phi_2),v_2(\phi_2)\Yleft_{\mathsf{G}}v_2(\phi_1))\\
&&
&v(\phi_1\weakrightarrow\phi_2)&=&(v_1(\phi_1)\!\rightarrow_{\mathsf{G}}\!v_1(\phi_2),v_1(\phi_1)\wedge v_2(\phi_2))
\end{array}\]
\end{definition}
\begin{remark}[Interpreting negations and (co-)implications]\label{rem:negandimpl}
$\psi\rightarrow\psi'$ is positively supported in $\L^2$ and $\mathsf{G}^2$ is interpreted as ‘positive evidence for $\psi$ is not stronger than for $\psi'$’. The negative support is obtained via co-implications. In the case $\L^2$, $\ominus$ measures the difference between negative supports of $\psi'$ and $\psi$. 
On the other hand, in $\mathsf{G}^2$, the negative support of $\psi\rightarrow\psi'$ is non-zero (and in fact is equal to the negative support of $\psi'$) when the negative support of $\psi'$ is stronger than that of $\psi$.

On the other hand, $\weakrightarrow$ in both $\mathsf{G}^2$ and $\L^2$ could be considered as being closer to the more intuitive ‘if \ldots, then \ldots’ in natural language. Thus, to obtain negative support of $\psi\weakrightarrow\psi'$, we use positive support of $\psi$ and negative support of $\psi'$. Falsity of $\weakrightarrow$ is thus more related to the traditional understanding of implication being false when the antecedent is true and the consequent is false.
\end{remark}
\begin{definition}[Validity and consequence] 
Let $\phi$ be a formula and $\Gamma$ a set of formulas of $\Luk^2$ (resp.\ $\mathsf{G}^2$) and
$v[\Gamma] \coloneqq 
\{v(\gamma) \mid \gamma \in \Gamma \}$.
\label{def:rectangle:validity}
\begin{itemize}[noitemsep,topsep=2pt]
    \item $\phi$ is $\Luk^2_{(x,y)}$-valid (resp.\ $\mathsf{G}^2_{(x,y)}$-valid) iff $\forall v:v(\phi)\in(x,y)^\uparrow$.
    \item $\Gamma\vDash_{\Luk^2_{(x,y)}}\phi$ (resp. $\Gamma\vDash_{\mathsf{G}^2_{(x,y)}}\phi$) iff  $\forall v:\text{if }v[\Gamma]\subseteq (x,y)^\uparrow\text{ then }v(\phi)\in(x,y)^\uparrow$.
\end{itemize}
\end{definition}
\begin{convention}
\label{conv:names:logics}
We introduce the following notation.
\begin{itemize}[noitemsep,topsep=2pt]
\item $\Luk^2_{(x,y)}(\rightarrow)$ stands for the $\Luk^2_{(x,y)}$ logics over $\{\mathbf{0},\neg,\wedge,\vee,\rightarrow\}$.
\item $\mathsf{G}^2_{(x,y)}(\rightarrow)$ stands for the $\mathsf{G}^2_{(x,y)}$ logics over $\{\mathbf{0},\mathbf{1},\neg,\wedge,\vee,\rightarrow,\Yleft\}$.
\item $\Luk^2_{(x,y)}(\weakrightarrow)$ stands for the $\Luk^2_{(x,y)}$ logics over $\{\mathbf{0},\neg,\wedge,\vee,\weakrightarrow\}$.
\item $\mathsf{G}^2_{(x,y)}(\weakrightarrow)$ stands for the $\mathsf{G}^2_{(x,y)}$ logics over $\{\mathbf{0},\mathbf{1},\neg,\wedge,\vee,\weakrightarrow\}$.
\end{itemize}
We note that $\mathbf{0}$ of $\Luk^2_{(x,y)}(\rightarrow)$ can be defined as $\neg(p\rightarrow p)$. However, since there is no definition of $\mathbf{0}$ using $\weakrightarrow$, we leave it in both languages for the sake of preserving the same tableau rules for all logics. Likewise, although $\mathbf{0}$ and $\mathbf{1}$ are definable in $\mathsf{G}^2_{(x,y)}(\rightarrow)$, their presence in the language simplifies the proofs of their semantical properties (cf. Propositions~\ref{no1no0} and~\ref{prop:no:lower:limits:for:G:neg}).
\end{convention}
\begin{remark}
\label{rk:conservative:extension}
Let $\phi$ be a formula over $\{0,\wedge,\vee,\supset\}$ with $\supset$ being the Boolean implication. Denote $\phi^\bullet$ the formula obtained from it by substituting $\supset$ for $\rightarrow$, and $\phi^\circ$   by substituting $\supset$ for $\weakrightarrow$. Since $v_1$'s behave precisely like the valuations in \L{}ukasiewicz (G\"{o}del) logic,  one can see that $\phi$ is \L{}-valid ($\mathsf{G}$-valid) iff $\phi^\bullet$ is $\Luk^2_{(1,0)}(\rightarrow)$-valid ($\mathsf{G}^2_{(1,0)}(\rightarrow)$-valid). Furthermore, $\phi$ is \L{}-valid ($\mathsf{G}$-valid) iff $\phi^\circ$ is $\Luk^2_{(1,1)}(\weakrightarrow)$-valid ($\mathsf{G}^2_{(1,1)}(\weakrightarrow)$-valid). Thus, $\Luk^2_{(1,0)}(\rightarrow)$ and $\Luk^2_{(1,1)}(\weakrightarrow)$ are conservative extensions of \L{} while $\mathsf{G}^2_{(1,0)}(\rightarrow)$ and $\mathsf{G}^2_{(1,1)}(\weakrightarrow)$ are conservative extensions of $\mathsf{G}$.
\end{remark}
\begin{remark}
Notice that if $v(p)=(1,1)$, then $v(p\weakrightarrow p)=(1,1)$ in $\mathsf{G}^2(\weakrightarrow)$ and $\L^2(\weakrightarrow)$. Thus, if we refuse to consider $(1,1)$ as a designated value, the weak implication ceases to be reflexive. 
Therefore,  $\Luk^2_{(x,y)}(\weakrightarrow)$'s and $\mathsf{G}^2_{(x,y)}(\weakrightarrow)$'s with sets of designated values  not containing $(1,1)$ do not extend $\L$ and $\mathsf{G}$.
\end{remark}
In order to work with extensions of $\L$ and $\mathsf{G}$, we are going to consider only $\Luk^2_{(x,y)}(\weakrightarrow)$'s and $\mathsf{G}^2_{(x,y)}(\weakrightarrow)$'s whose sets of designated values extend $(1,1)^\uparrow$, that is $\Luk^2_{(x,1)}(\weakrightarrow)$ and $\mathsf{G}^2_{(x,1)}(\weakrightarrow)$.
In the remainder of the article $\phi,\varphi, \chi, \psi$ denote formulas. Unless there is some ambiguity, we do not specify to which language they belong.
\subsection{Semantical properties of $\Luk^2_{(x,y)}(\rightarrow)$}
\label{ssec:sem:properties:L2}
In this section, we are going to explore how the choice of $(x,y)^\uparrow$ affects the set of $\Luk^2_{(x,y)}(\rightarrow)$-valid formulas. In particular, we are providing families of formulas differentiating different $\Luk^2_{(x,y)}(\rightarrow)$-validities.
\begin{definition}[Closure under conflation]
\label{def:closure:conflation}
We say that a filter $D$ of $[0,1]\odot[0,1]$ is \emph{closed under conflation} if for any $(x,y)\in D$, we have $(1-y,1-x)\in D$.
\end{definition}
In bilattices, the negation $\neg$ corresponds to a symmetry w.r.t.\ the horizontal axis and the conflation corresponds to a symmetry w.r.t.\ the vertical axis. Notice that a filter
$(x,y)^\uparrow$ is closed under conflation iff $y=1-x$.
In $\Luk^2_{(x,y)}(\rightarrow)$, conflation can be defined as $\neg{\sim}$ or equi\-va\-lently~${\sim}\neg$ (cf. Figure~\ref{fig:square:with:filter}, right).
\begin{proposition}
\label{prop:rectangles:to:squares}
\begin{itemize}
\item[]
\item Let $y\geq 1-x$. Then $\phi$ is $\Luk^2_{(x,y)}(\rightarrow)$-valid iff $\phi$ is $\Luk^2_{(x,1-x)}(\rightarrow)$-valid.
\item Let $y<1-x$. Then $\phi$ is $\Luk^2_{(x,y)}(\rightarrow)$-valid iff $\phi$ is $\Luk^2_{(1-y,y)}(\rightarrow)$-valid.
\end{itemize}
\end{proposition}
The following statements show that by choosing different sets of designated values, we can alter the sets of tautologies.
\begin{proposition}\label{smallsquares}
Let $m,n\in\{2,3,\ldots\}$. Then  $\Luk^2_{\left(\frac{m-1}{m},\frac{1}{m}\right)}\subsetneq\L^2_{\left(\frac{n-1}{n},\frac{1}{n}\right)}$ iff $m>n$.
\end{proposition}
Note, however, that while $\L_{\left(\frac{1}{2},\frac{1}{2}\right)}$ validates $p\vee{\sim}p$, it does not collapse into classical logic as the following propositions show.
\begin{proposition}\label{bigsquares}
Let $m,n\in\{3,4,\ldots\}$. Then $\L^2_{\left(\frac{m-2}{2m},\frac{m+2}{2m}\right)}\!\subsetneq\!\L^2_{\left(\frac{n-2}{2n},\frac{n+2}{2n}\right)}$ iff $m\!>\!n$.
\end{proposition}
We end this subsection by noting that all $\Luk^2_{(x,y)}(\rightarrow)$'s where $(x,y)^\uparrow$ is prime are paraconsistent in the following sense: $p,\neg p\nvDash_{\Luk^2_{(x,y)}(\rightarrow)}q$. 
Furthermore, if $\left(\frac{1}{2},\frac{1}{2}\right)\in(x,y)^\uparrow$, the logic is paraconsistent even w.r.t. $\sim$ since $p,{\sim}p\nvDash_{\Luk^2_{(x,y)}(\rightarrow)}q$. 
Last but not least, most $\Luk^2_{(x,y)}(\rightarrow)$'s are not closed under modus ponens.
\begin{proposition}\label{nomodusponens}
Let $\Luk^2_{(1,0)}(\rightarrow)\subsetneq\Luk^2_{(x,y)}(\rightarrow)$. Then $\Luk^2_{(x,y)}(\rightarrow)$ is not closed under modus ponens.
\end{proposition}
\subsection{Semantical properties of $\mathsf{G}^2(\rightarrow)$}
\label{ssec:sem:properties:G2}
In this section, we show that all $\mathsf{G}^2_{(x,y)}(\rightarrow)$ logics have the same set of valid formulas. This means that just as the original G\"{o}del logic, $\mathsf{G}^2(\rightarrow)$ can be seen as the logic of comparative truth. Furthermore, the presence of the second dimension allows to interpret $\mathsf{G}^2(\rightarrow)$ as the logic of comparative truth and falsehood.
\begin{proposition}
\label{no1no0}
Let $\phi$ be a formula over $\{\mathbf{0},\mathbf{1},\neg,\wedge,\vee,\rightarrow,\Yleft\}$. For any $v(p)=(x,y)$ let $v^*(p)=(1-y,1-x)$. Then $v(\phi)=(x,y)$ iff $v^*(\phi)=(1-y,1-x)$.
\end{proposition}
\begin{proposition}
\label{prop:no:lower:limits:for:G:neg}
Let $\phi$ be a formula over $\{\mathbf{0},\mathbf{1},\neg,\wedge,\vee,\rightarrow,\Yleft\}$  such that $v(\phi)\geq (x,y)$ for any $v$ and some fixed $(x,y)\neq(0,1)$. Then $v'(\phi)=(1,0)$ for any $v'$.
\end{proposition}
The last three propositions show that in contrast to $\mathsf{L}^2(\rightarrow)$, the changing of the set of designated values does not change the set of valid formulas as long as the set remains a filter\footnote{Notice that $p\vee\neg p$ would be valid for $D=[0,1]\odot[0,1]\setminus\{(0,1)\}$. But $D$ is not a~filter.} on $[0,1]\odot[0,1]$ generated by a single point. However, while the sets of tautologies remain the same, the entailment relation can be made paraconsistent. Indeed, it suffices to choose any prime $(x,y)^\uparrow$ and the entailment ceases to be explosive in the following sense: $p,\neg p\nvDash_{\mathsf{G}^2_{(x,y)}}q$.

Furthermore, the propositions have an important corollary which simplifies the construction of the tableaux proofs.
\begin{corollary}
\label{cor:not1=not0}
$v(\phi)=(1,0)$ for any $v$ iff $v'_1(\phi)=1$ for any $v'$.
\end{corollary}
\section{Tableaux}
\label{sec:tableaux}
First, we give a general definition of a constraint tableaux, then in sections~\ref{L2constraint} and~\ref{G2constraint}, we introduce tableaux for $\Luk^2$'s and $\mathsf{G}^2$'s.

\begin{definition}[Constraint tableaux]
\label{def:constraint:tableaux} Let $\mathsf{Label}$ be a set of labels and $\mathcal{L}$ a set of formulas. A \emph{constraint} is one of these three expressions:
\begin{itemize}[noitemsep,topsep=2pt]
\item \emph{Labelled formulas} of the form ${L}:\phi$ with ${L} \in \mathsf{Label}$ and $\phi \in \mathcal{L}$,
\item \emph{Numerical constraints} of the form $c\leq d$ or $c<d$ with $c,d\in[0,1]$,
\item \emph{Formulaic constraints} of the form ${L}:\phi\leqslant{L}':\phi'$ or ${L}:\phi<{L}':\phi'$  with $L,L'\in\mathsf{Label}$ and $\phi,\phi'\in\mathcal{L}$.
\end{itemize}
A \emph{constraint tableau} is a~downward branching tree each branch of which is a non-empty set of constraints.
Each branch $\mathcal{B}$ can be extended by applications of a given set of rules. If no rule application adds new entries to $\mathcal{B}$, it is called \emph{complete}.
\end{definition}
As expected, in labelled formulas, $L$ is some \emph{set of values}. Thus, the intended interpretation of $L:\phi$ is ‘$\phi$ has some value from $L$’. In formulaic constraints, $L$ and $L'$ are \emph{components of $\phi$'s valuation}. Hence, the intended interpretation of ${L}:\phi\leqslant{L}':\phi'$ is ‘the component of $\phi$'s valuation denoted by $L$ is less or equal to the component of $\phi'$'s valuation denoted by $L'$’. The detailed interpretations of all types of entries for each tableau calculus are given in definitions~\ref{def:L2constrainttableau} and~\ref{def:G2constrainttableau} as well as remarks~\ref{rem:TL2meaning} and~\ref{rem:TG2meaning}.

Henceforth, we only state the rules and the closure conditions for branches. In what follows, we identify a branch with the set of entries that appear at some point on the branch.

Since our logics are hybrids between $\FDE$ and \L{} (or $\mathsf{G}$), we can combine the constraint tableaux framework with the $\FDE$-tableaux by D'Agostino~\cite{DAgostino1990}. In particular, it means that we  use two kinds of labelled formulas and formulaic constraints: those that concern the left coordinate (evidence for the statement) and those that concern the right coordinate (evidence against the statement).
\subsection{Constraint tableaux for $\Luk^2$}\label{L2constraint}
\begin{definition}[Constraint tableau for $\Luk^2$ --- $\mathcal{T}\left(\Luk^2_{(x,y)}\right)$]\label{def:L2constrainttableau}
Branches contain \emph{labelled formulas} of the form $\phi\leqslant_1i$, $\phi\leqslant_2 i$, $\phi\geqslant_1i$, or $\phi\geqslant_2i$, and \emph{numerical constraints} of the form $i\leq j$ with $i,j \in [0,1]$. We call \emph{atomic labelled formulas} labelled formulas where $\phi \in \mathsf{Prop}$.

Each branch can be extended by an application of one of the  rules in Figure~\ref{L2rulesfigure} where $i,j \in [0,1]$.
\begin{figure}
\centering
\[\begin{array}{cccc}
\mathbf{0}\leqslant_1\dfrac{\mathbf{0}\leqslant_1i}{0\leq i}\quad \quad
&
\mathbf{0}\leqslant_2\dfrac{\mathbf{0}\leqslant_2i}{1\leq i}\quad\quad
&
\mathbf{0}\geqslant_1\dfrac{\mathbf{0}\geqslant_1i}{0\geq i}
\quad\quad
&
\mathbf{0}\geqslant_2\dfrac{\mathbf{0}\geqslant_2i}{1\geq i}
\\
&\\
\neg\leqslant_1\dfrac{\neg\phi\leqslant_1i}{\phi\leqslant_2i}\quad \quad
&
\neg\leqslant_2\dfrac{\neg\phi\leqslant_2i}{\phi\leqslant_1i}\quad\quad
&
\neg\geqslant_1\dfrac{\neg\phi\geqslant_1i}{\phi\geqslant_2i}
\quad\quad
&
\neg\geqslant_2\dfrac{\neg\phi\geqslant_2i}{\phi\geqslant_1i}
\end{array}\]
\[\begin{array}{cc}
\rightarrow\leqslant_1\dfrac{\phi_1\rightarrow\phi_2\leqslant_1i}{i\geq1\left|\begin{matrix}\phi_1\geqslant_11-i+j\\\phi_2\leqslant_1j\\j\leq i\end{matrix}\right.}
\quad \quad
&
\rightarrow\leqslant_2\dfrac{\phi_1\rightarrow\phi_2\leqslant_2i}{\begin{matrix}\phi_1\geqslant_2j\\\phi_2\leqslant_2i+j\end{matrix}}\\
&\\
\rightarrow\geqslant_1\dfrac{\phi_1\rightarrow\phi_2\geqslant_1 i}{\begin{matrix}\phi_1\leqslant_11-i+j\\\phi_2\geqslant_1j\end{matrix}}
\quad \quad
&
\rightarrow\geqslant_2\dfrac{\phi_1\rightarrow\phi_2\geqslant_2i}{i\leq0\left|\begin{matrix}\phi_1\leqslant_2j\\\phi_2\geqslant_2i+j\\j\leq 1-i\end{matrix}\right.}\\
&\\
\weakrightarrow\leqslant_1\dfrac{\phi_1\weakrightarrow\phi_2\leqslant_1i}{i\geq1\left|\begin{matrix}\phi_1\geqslant_11-i+j\\\phi_2\leqslant_1j\\j\leq  i\end{matrix}\right.}&\weakrightarrow\leqslant_2\dfrac{\phi_1\weakrightarrow\phi_2\leqslant_2i}{\begin{matrix}\phi_1\leqslant_2i+j\\\phi_2\leqslant_11-j\end{matrix}}\\
&\\
\weakrightarrow\geqslant_1\dfrac{\phi_1\weakrightarrow\phi_2\geqslant_1 i}{\begin{matrix}\phi_1\leqslant_11-i+j\\\phi_2\geqslant_1j\end{matrix}}&\weakrightarrow\geqslant_2\dfrac{\phi_1\weakrightarrow\phi_2\geqslant_2i}{i\leq0\left|\begin{matrix}\phi_1\geqslant_2i+j\\\phi_2\geqslant_11-j\\j\leq 1-i\end{matrix}\right.}\\
&\\
\wedge\leqslant_1\dfrac{\phi_1\wedge\phi_2\leqslant_1i}{\phi_1\leqslant_1i\mid\phi_2\leqslant_1i}&\wedge\leqslant_2\dfrac{\phi_1\wedge\phi_2\leqslant_2i}{\begin{matrix}\phi_1\leqslant_2i\\\phi_2\leqslant_2i\end{matrix}}\\
&\\
\wedge\geqslant_1\dfrac{\phi_1\wedge\phi_2\geqslant_1i}{\begin{matrix}\phi_1\geqslant_1i\\\phi_2\geqslant_1i\end{matrix}}&\wedge\geqslant_2\dfrac{\phi_1\wedge\phi_2\geqslant_2i}{\phi_1\geqslant_2i\mid\phi_2\geqslant_2i}\\
&\\
\vee\leqslant_1\dfrac{\phi_1\vee\phi_2\leqslant_1i}{\begin{matrix}\phi_1\leqslant_1i\\\phi_2\leqslant_1i\end{matrix}}&\vee\leqslant_2\dfrac{\phi_1\vee\phi_2\leqslant_2i}{\phi_1\leqslant_2i\mid\phi_2\leqslant_2i}\\
&\\
\vee\geqslant_1\dfrac{\phi_1\vee\phi_2\geqslant_1i}{\phi_1\geqslant_1i\mid\phi_2\geqslant_1i}&\vee\geqslant_2\dfrac{\phi_1\vee\phi_2\geqslant_2i}{\begin{matrix}\phi_1\geqslant_2i\\\phi_2\geqslant_2i\end{matrix}}
\end{array}\]
\caption{Rules of $\mathcal{T}\left(\Luk^2_{(x,y)}\right)$. Vertical bars denote splitting of the branch.}
\label{L2rulesfigure}
\end{figure}
Let $i$'s be in $[0,1]$ and $x$'s be  variables ranging over the real interval $[0,1]$. We define the translation $\tau$ from labelled formulas to linear inequalities as follows:
$$\tau(\phi\!\leqslant_1\!i)=x_\phi^L\!\leq\!i; \; \tau(\phi\!\geqslant_1\!i)=x_\phi^L\!\geq\!i; \; \tau(\phi\!\leqslant_2\!i)=x_\phi^R\leq i; \; \tau(\phi\!\geqslant_2\!i)=x_\phi^R\!\geq\!i$$
Let $\bullet\in\{\leqslant_1,\geqslant_1\}$ and $\circ\in\{\leqslant_2,\geqslant_2\}$. A tableau branch
$$\mathcal{B}=\{\phi_1\circ i_1,\ldots,\phi_m\circ i_m,\phi'_1\bullet j_1,\ldots,\phi'_n\bullet j_n,k_1\leq l_1,\ldots,k_q\leq l_q\}$$
is \emph{closed} if the system of inequalities
\[\tau(\phi_1\circ i_1),\ldots,\tau(\phi_m\circ i_m),\tau(\phi'_1\bullet j_1),\ldots,\tau(\phi'_n\bullet j_n),k_1\leq l_1,\ldots,k_q\leq l_q\]
does not have solutions. Otherwise, $\mathcal{B}$ is \emph{open}. A tableau is \emph{closed} if all its branches are closed.

$\phi$ has a \emph{$\mathcal{T}\left(\Luk^2_{(x,y)}\right)$ proof} if the tableaux beginning with $\{\phi\leqslant_1c, c<x\}$ and $\{\phi\geqslant_2d,d>y\}$ are both closed.
\end{definition}
\begin{remark}[How to interpret the rules of $\mathcal{T}\left(\Luk^2_{(x,y)}\right)$?]\label{rem:TL2meaning} Consider for instance the rule $\rightarrow\leqslant_2$. It's meaning is: $v_2(\phi_1\rightarrow\phi_2)\leq  i$ iff there is $j \in [0,1]$ s.t. 
$v_2(\phi_1)\geq  j$ and $v_2(\phi_2)\leq  i+j$. While rule $\wedge\!\!\leqslant_1$ means 
$v_1(\phi_1\wedge\phi_2)\leq  i$ iff 
either $v_1(\phi_1)\leq  i$ or $v_1(\phi_2)\leq  i$.
\end{remark}
To prove completeness and soundness, we need the following definitions.
\begin{definition}[Satisfying valuation of a branch]
Let $v$ be a valuation and $\mathrm{k}\in\{1,2\}$. 
$v$~\emph{satisfies} a labelled formula $\phi\leqslant_\mathrm{k}i$ (resp.\ $\phi\geqslant_\mathrm{k}i$) iff $v_\mathrm{k}(\phi)\leq  i$ (resp.\ $v_\mathrm{k}(\phi)\geq  i$).
$v$~\emph{satisfies} a branch $\mathcal{B}$ iff $v$~\emph{satisfies} any labelled formula in $\mathcal{B}$.
A~branch $\mathcal{B}$ is \emph{satisfiable} iff there is a valuation which satisfies it. 
\end{definition}
\begin{theorem}[Soundness and completeness]\label{tableauxsoundness}\label{tableauxcompleteness}
$\phi$ is $\Luk^2_{(x,y)}(\rightarrow)$-valid (resp.\ $\Luk^2_{(x,y)}(\weakrightarrow)$-valid) iff there is a $\mathcal{T}\left(\Luk^2_{(x,y)}\right)$ proof for it.
\end{theorem}
\begin{proof}
The soundness follows from the fact that no closed branch is realisable and that if a premise of the rule is realisable, then all labelled formulas are satisfied in at least one of the conclusions.

To show completeness, we proceed by contraposition. We need to show that complete open branches are satisfiable.

Assume that $\mathcal{B}$ is a complete open branch. We construct the satisfying valuation as follows. Let $*\in\{\leqslant_1,\geqslant_1,\leqslant_2,\geqslant_2\}$ and $p_1, \ldots, p_m$ be the propositional variables appearing in the atomic labelled formulas in $\mathcal{B}$.
Let $\{p_1*i_1,\ldots,p_m*i_n\}$ and $\{k_1\leq  l_1,\ldots,k_q\leq  l_q\}$ be the sets of all atomic labelled formulas and all numerical constraints  in $\mathcal{B}$. Notice that one variable might appear in many atomic labelled formulas, hence we might have $m \neq n$. Since $\mathcal{B}$ is complete and open, the following system of linear inequalities over the set of variables $\{x_{p_1}^L, x_{p_1}^R, \ldots, x_{p_m}^L, x_{p_m}^R  \}$ must have at least one solution under the constrains listed:
\[ \tau(p_1*i_1),\ldots,\tau(p_m*i_n), k_1\leq  l_1,\ldots,k_q\leq  l_q.\]
Let $c=(c^L_1,c^R_1,\ldots,c^L_m,c^R_m)$
be a solution to the above system of inequalities such that $c^L_j$ (resp.\ $c^R_j$) is the value of $x_{p_j}^L$ (resp.\ $x_{p_j}^R$). Define the valuation $v$ as follows: $v(p_j)=(c^L_j,c^R_j)$.

It remains to show by induction on $\phi$ that all formulas present at $\mathcal{B}$ are satisfied by $v$. The basis case of variables holds by construction of $v$. We  consider only the most instructive case of $\phi_1\rightarrow\phi_2\geqslant_2i$ as the other cases are straightforward.

Assume that $\phi_1\rightarrow\phi_2\geqslant_2i\in\mathcal{B}$. Then, by completeness of $\mathcal{B}$, either $i\leq0\in\mathcal{B}$, in which case, $\phi_1\rightarrow\phi_2\geqslant_2i$ is trivially satsified, or  $\phi_1\leqslant_2j,\phi_2\geqslant_2i+j\in\mathcal{B}$. Furthermore, by the induction hypothesis, $v$ satisfies $\phi_1\leqslant_2j$ and $\phi_2\geqslant_2i+j$, and we also have that $j\leq1-i$.  Now, to show that $v$ satisfies $\phi_1\rightarrow\phi_2\geqslant_2i$, recall from semantics that $v_2(\phi_1\rightarrow\phi_2)=\max(0,v_2(\phi_2)-v_2(\phi_1))$.

Now, we have
\begin{align*}
\max(0,v_2(\phi_2)-v_2(\phi_1))&\geq \max(0,i+j-j)=\max(0,i)= i
\end{align*}
as desired.

The cases of other connectives can be tackled in a similar fashion.
\hfill $\Box$ 
\end{proof}
\subsection{Constraint tableaux for $\mathsf{G}^2$}\label{G2constraint}
\begin{definition}[Constraint tableaux for $\mathsf{G}^2$ --- $\mathcal{T}(\mathsf{G}^2)$]\label{def:G2constrainttableau}
Let $\lesssim~\in\{<,\leqslant\}$ and $\gtrsim~\in\{<,\leqslant\}$.
Branches contain:
\begin{itemize}[noitemsep,topsep=2pt]
\item \emph{formulaic constraints} of the form $\mathbf{x}:\phi\lesssim\mathbf{x}':\phi'$ 
with $\mathbf{x}\in\{1,2\}$;
\item \emph{numerical constraints} of the form ${c}\lesssim{c}'$ with ${c},{c}'\in\{{1},{0}\}$;
\item \emph{labelled formulas} of the form $\mathbf{x}:\phi*{c}$ with $*\in\{\lesssim,\gtrsim\}$.
\end{itemize}
We abbreviate all these types of entries with $\mathfrak{X}\lesssim\mathfrak{X}'$. Each branch can be extended by an application of one of the  rules in Figure~\ref{G2rulesfigure} where 
$\mathbf{c}\neq\mathbf{c}'$, $c\neq c'$, $\mathbf{c},\mathbf{c}'\in\{\mathbf{0},\mathbf{1}\}$  and $c,c'\in\{0,1\}$.

\begin{figure}
\centering
\[\begin{array}{cccc}
\mathbf{c}_1\!\lesssim\!\dfrac{1\!:\!\mathbf{c}\!\lesssim\!\mathfrak{X}}{c\!\lesssim\!\mathfrak{X}}
& \qquad
\mathbf{c}_2\!\lesssim\!\dfrac{2\!:\!\mathbf{c}\!\lesssim\!\mathfrak{X}}{c'\!\lesssim\!\mathfrak{X}}
& \qquad
\mathbf{c}_1\!\gtrsim\!\dfrac{1\!:\!\mathbf{c}\!\gtrsim\!\mathfrak{X}}{c\!\gtrsim\!\mathfrak{X}}
& \qquad
\mathbf{c}_2\!\gtrsim\!\dfrac{2\!:\!\mathbf{c}\!\gtrsim\!\mathfrak{X}}{c'\!\gtrsim\!\mathfrak{X}}\\
&\\
\neg_1\!\lesssim\!\dfrac{1\!:\!\neg\phi\!\lesssim\!\mathfrak{X}}{2\!:\!\phi\!\lesssim\!\mathfrak{X}}
& \qquad
\neg_2\!\lesssim\!\dfrac{2\!:\!\neg\phi\!\lesssim\!\mathfrak{X}}{1\!:\!\phi\!\lesssim\!\mathfrak{X}}
& \qquad
\neg_1\!\gtrsim\!\dfrac{1\!:\!\neg\phi\!\gtrsim\!\mathfrak{X}}{2\!:\!\phi\!\gtrsim\!\mathfrak{X}}
& \qquad
\neg_2\!\gtrsim\!\dfrac{2\!:\!\neg\phi\!\gtrsim\!\mathfrak{X}}{1\!:\!\phi\!\gtrsim\!\mathfrak{X}}
\end{array}\]
\smallskip
\[\begin{array}{cccc}
\wedge_1\!\gtrsim\!\dfrac{1\!:\!\phi\!\wedge\!\phi'\!\gtrsim\!\mathfrak{X}}{\begin{matrix}1\!:\!\phi\!\gtrsim\!\mathfrak{X}\\1\!:\!\phi'\!\gtrsim\!\mathfrak{X}\end{matrix}}
&~
\wedge_2\!\lesssim\!\dfrac{2\!:\!\phi\!\wedge\!\phi'\!\lesssim\!\mathfrak{X}}{\begin{matrix}2\!:\!\phi\!\lesssim\!\mathfrak{X}\\2\!:\!\phi'\!\lesssim\!\mathfrak{X}\end{matrix}}
&~
\vee_1\!\lesssim\!\dfrac{1\!:\!\phi\!\vee\!\phi'\!\lesssim\!\mathfrak{X}}{\begin{matrix}1\!:\!\phi\!\lesssim\!\mathfrak{X}\\1\!:\!\phi'\!\lesssim\!\mathfrak{X}\end{matrix}}
&~
\vee_2\!\gtrsim\!\dfrac{2\!:\!\phi\!\vee\!\phi'\!\gtrsim\!\mathfrak{X}}{\begin{matrix}2\!:\!\phi\!\gtrsim\!\mathfrak{X}\\2\!:\!\phi'\!\gtrsim\!\mathfrak{X}\end{matrix}}
\end{array}\]
\[\begin{array}{cc}
\wedge_1\!\lesssim\!\dfrac{1\!:\!\phi\wedge\phi'\!\lesssim\!\mathfrak{X}}{1\!:\!\phi\!\lesssim\!\mathfrak{X}\mid1\!:\!\phi'\!\lesssim\!\mathfrak{X}}
&\quad
\wedge_2\!\gtrsim\!\dfrac{2\!:\!\phi\wedge\phi'\!\gtrsim\!\mathfrak{X}}{2\!:\!\phi\!\gtrsim\!\mathfrak{X}\mid2\!:\!\phi'\!\gtrsim\!\mathfrak{X}}\\
&\\
\vee_1\!\gtrsim\!\dfrac{1\!:\!\phi\vee\phi'\!\gtrsim\!\mathfrak{X}}{1\!:\!\phi\!\gtrsim\!\mathfrak{X}\mid1\!:\!\phi'\!\gtrsim\!\mathfrak{X}}
& \qquad
\vee_2\!\lesssim\!\dfrac{2\!:\!\phi\vee\phi'\!\lesssim\!\mathfrak{X}}{2\!:\!\phi\!\lesssim\!\mathfrak{X}\mid2\!:\!\phi'\!\lesssim\!\mathfrak{X}}\\
\end{array}\]
\smallskip
\[\begin{array}{ccc}
\rightarrow_1\!\leqslant\!\dfrac{1\!:\!\phi\rightarrow\phi'\!\leqslant\!\mathfrak{X}}{\mathfrak{X}\!\geqslant\!{1}\left|\begin{matrix}\mathfrak{X}\!<\!{1}\\1\!:\!\phi'\!\leqslant\!\mathfrak{X}\\1\!:\!\phi\!>\!1\!:\!\phi'\end{matrix}\right.}&\rightarrow_1\!\gtrsim\!\dfrac{1\!:\!\phi\rightarrow\phi'\!\gtrsim\!\mathfrak{X}}{1\!:\!\phi\!\leqslant\!1\!:\!\phi'\mid1\!:\!\phi'\!\gtrsim\!\mathfrak{X}}&\rightarrow_1\!<\!\dfrac{1\!:\!\phi\rightarrow\phi'\!<\!\mathfrak{X}}{\begin{matrix}1\!:\!\phi'\!<\!\mathfrak{X}\\1\!:\!\phi\!>\!1\!:\!\phi'\end{matrix}}\\
&\\
\rightarrow_2\!\lesssim\!\dfrac{2\!:\!\phi\rightarrow\phi'\!\lesssim\!\mathfrak{X}}{2\!:\!\phi'\!\leqslant\!2\!:\!\phi\mid2\!:\!\phi'\!\lesssim\!\mathfrak{X}}&\rightarrow_2\!\geqslant\!\dfrac{2\!:\!\phi\rightarrow\phi'\!\geqslant\!\mathfrak{X}}{\mathfrak{X}\!\leqslant\!{0}\left|\begin{matrix}\mathfrak{X}\!>\!{0}\\2\!:\!\phi'\!\geqslant\!\mathfrak{X}\\2\!:\!\phi'\!>\!2\!:\!\phi\end{matrix}\right.}&\rightarrow_2\!>\!\dfrac{2\!:\!\phi\rightarrow\phi'\!>\!\mathfrak{X}}{\begin{matrix}2\!:\!\phi'\!>\!\mathfrak{X}\\2\!:\!\phi'\!>\!2\!:\!\phi\end{matrix}}\\
&\\
\Yleft_1\!\lesssim\!\dfrac{1\!:\!\phi\!\Yleft\!\phi'\!\lesssim\!\mathfrak{X}}{1\!:\!\phi\!\leqslant\!1\!:\!\phi'\mid1\!:\!\phi\!\lesssim\!\mathfrak{X}}&\Yleft_1\!>\!\dfrac{1\!:\!\phi\!\Yleft\!\phi'\!>\!\mathfrak{X}}{\begin{matrix}1\!:\!\phi\!>\!\mathfrak{X}\\1\!:\!\phi\!>\!1\!:\!\phi'\end{matrix}}&\Yleft_1\!\geqslant\!\dfrac{1\!:\!\phi\!\Yleft\!\phi'\!\geqslant\!\mathfrak{X}}{\mathfrak{X}\!\leqslant\!{0}\left|\begin{matrix}\mathfrak{X}\!>\!{0}\\1\!:\!\phi\!\geqslant\!\mathfrak{X}\\1\!:\!\phi\!>\!1\!:\!\phi'\end{matrix}\right.}\\
&\\
\Yleft_2\!\gtrsim\!\dfrac{2\!:\!\phi\Yleft\phi'\!\gtrsim\!\mathfrak{X}}{2\!:\!\phi\!\gtrsim\!\mathfrak{X}\mid2\!:\!\phi'\!\leqslant\!2\!:\!\phi}&\Yleft_2\!\leqslant\!\dfrac{1\!:\!\phi\Yleft\phi'\!\leqslant\!\mathfrak{X}}{\mathfrak{X}\!\geqslant\!{1}\left|\begin{matrix}\mathfrak{X}\!<\!{1}\\2\!:\!\phi\!\leqslant\!\mathfrak{X}\\2\!:\!\phi'\!>\!2\!:\!\phi\end{matrix}\right.}&\Yleft_2\!<\!\dfrac{2\!:\!\phi\Yleft\phi'\!<\!\mathfrak{X}}{\begin{matrix}2\!:\!\phi\!<\!\mathfrak{X}\\2\!:\!\phi\!<\!2\!:\!\phi'\end{matrix}}\\
&\\
\weakrightarrow_1\!\leqslant\!\dfrac{1\!:\!\phi\!\weakrightarrow\!\phi'\!\leqslant\!\mathfrak{X}}{\mathfrak{X}\!\geqslant\!{1}\left|\begin{matrix}\mathfrak{X}\!<\!{1}\\1\!:\!\phi'\!\leqslant\!\mathfrak{X}\\1\!:\!\phi\!>\!1\!:\!\phi'\end{matrix}\right.}&\weakrightarrow_1\!\gtrsim\!\dfrac{1\!:\!\phi\!\weakrightarrow\!\phi'\!\gtrsim\!\mathfrak{X}}{1\!:\!\phi\!\leqslant\!1\!:\!\phi'\mid1\!:\!\phi'\!\gtrsim\!\mathfrak{X}}&\weakrightarrow_1\!<\!\dfrac{1\!:\!\phi\!\weakrightarrow\!\phi'\!<\!\mathfrak{X}}{\begin{matrix}1\!:\!\phi'\!<\!\mathfrak{X}\\1\!:\!\phi\!>\!1\!:\!\phi'\end{matrix}}
\end{array}\]
\smallskip
\[\begin{array}{cc}
\weakrightarrow_2\lesssim\dfrac{2\!:\!\phi\weakrightarrow\phi'\!\lesssim\!\mathfrak{X}}{1\!:\!\phi\!\lesssim\!\mathfrak{X}\mid2\!:\!\phi'\!\lesssim\!\mathfrak{X}}
& \qquad
\weakrightarrow_2\gtrsim\dfrac{2\!:\!\phi\weakrightarrow\phi'\!\gtrsim\!\mathfrak{X}}{\begin{matrix}1\!:\!\phi\!\gtrsim\!\mathfrak{X}\\2\!:\!\phi'\!\gtrsim\!\mathfrak{X}\end{matrix}}
\end{array}\]
\caption{Rules of $\mathcal{T}\left(\mathsf{G}^2\right)$. Vertical bars denote branching; $\mathbf{c}\neq\mathbf{c}'$, $c\neq c'$, $\mathbf{c},\mathbf{c}'\in\{\mathbf{0},\mathbf{1}\}$, $c,c'\in\{0,1\}$.}
\label{G2rulesfigure}
\end{figure}

A tableau's branch $\mathcal{B}$ is \emph{closed} iff at least one of the following conditions applies:
\begin{itemize}[noitemsep,topsep=2pt]
\item the transitive closure of $\mathcal{B}$ under $\lesssim$ contains $\mathfrak{X}<\mathfrak{X}$,
\item ${0}\geqslant{1}\in\mathcal{B}$ or $\mathfrak{X}>{1}\in\mathcal{B}$ or $\mathfrak{X}<{0}\in\mathcal{B}$.
\end{itemize}
A tableau is \emph{closed} iff all its branches are closed. We say that there is a \emph{tableau proof} of $\phi$ iff there is a closed tableau starting from $1:\phi<{1}$.
\end{definition}
\begin{remark}[Interpretation of constraints]\label{rem:TG2meaning}
Formulaic constraint $\mathbf{x}:\phi\leqslant\mathbf{x}':\phi'$ encodes the fact that $v_\mathbf{x}(\phi) \leq v_{\mathbf{x}'}(\phi')$, similarly labelled formula $\mathbf{x}:\phi \leqslant {c}$ encodes the fact that $v_\mathbf{x}(\phi) \leq c$.
\end{remark}
\begin{definition}[Satisfying valuation of a branch]\label{G2branchsatisfaction}
Let $\mathbf{x}, \mathbf{x}'\in\{1,2\}$. Branch $\mathcal{B}$ is \emph{satisfied} by a valuation $v$ iff
\begin{itemize}[noitemsep,topsep=2pt]
\item $v_\mathbf{x}(\phi)\leq v_{\mathbf{x}'}(\phi')$ for any $\mathbf{x}:\phi\leqslant\mathbf{x}':\phi'\in\mathcal{B}$ and
\item $v_\mathbf{x}(\phi)\leq c$ for any $\mathbf{x}:\phi\leqslant{c}\in\mathcal{B}$ s.t.\ ${c}\in\{{0},{1}\}$.
\end{itemize}
\end{definition}
\begin{theorem}[Soundness and completeness]\label{Gconstraintcompleteness}
$\phi$ is $\mathsf{G}^2$-valid iff it has a $\mathcal{T}(\mathsf{G}^2)$ proof.
\end{theorem}
\begin{proof}
For soundness, we check that if the premise of the rule is satisfied, then so is at least one of its conclusions.

For completeness, we show that every complete open branch $\mathcal{B}$ is satisfiable.
We construct the satisfying valuation as follows. If $\mathbf{x}:p
\geqslant1\in\mathcal{B}$, we set $v_1(p)=1$. If $1:p\leqslant0\in\mathcal{B}$, we set $v_1(p)=0$. 
We do likewise for $2:p\leqslant0$ and $2:p\geqslant1$. To set the values of the remaining variables $q_1$, \ldots, $q_n$, we proceed as follows. Denote $\mathcal{B}^+$ the transitive closure of $\mathcal{B}$ under $\lesssim$ and let
\[[\mathbf{x}:q_i]=\left\{\mathbf{x}':q_j \; \left| \; \begin{matrix}(\mathbf{x}:q_i\leqslant\mathbf{x}':q_j\in\mathcal{B}^+\text{ or }
\mathbf{x}:q_i\geqslant\mathbf{x}':q_j\in\mathcal{B}^+)\\
\text{and}\\
\mathbf{x}:q_i<\mathbf{x}':q_j
\notin\mathcal{B}^+
\text{ and } \mathbf{x}:q_i>\mathbf{x}':q_j\notin\mathcal{B}^+
\end{matrix}\right.\right\}\]
It is clear that there are at most $2n$ $[\mathbf{x}:q_i]$'s since the only possible loop in $\mathcal{B}^+$ is $\mathbf{x}:r\leqslant\ldots\leqslant\mathbf{x}:r$, but in such a loop all elements belong to $[\mathbf{x}:r]$. We put $[\mathbf{x}:q_i]\preceq[\mathbf{x}':q_j]$ iff there are $\mathbf{x}:r\in[\mathbf{x}:q_i]$ and $\mathbf{x}':r'\in[\mathbf{x}':q_j]$ s.t. $\mathbf{x}:r\leqslant\mathbf{x}':r'\in\mathcal{B}^+$.

We now set the valuation of these variables as follows
\begin{equation}
v_\mathbf{x}(q_i)=\dfrac{|\{[\mathbf{x}':q']\mid[\mathbf{x}':q']\preceq[\mathbf{x}:q_i]\}|}{2n}\tag{$*$}\label{valuationchain}
\end{equation}
Thus, all constraints containing only variables are satisfied.

It remains to show that all other constraints are satisfied. For that, we prove that if at least one conclusion of the rule is satisfied, then so is the premise. We consider only the case of $\weakrightarrow_2\lesssim$. Let $1:\phi_1\lesssim\mathfrak{X}$ be satisfied. W.l.o.g., assume that $\mathfrak{X}=2:\psi$ and $\lesssim=<$. Thus, $v_1(\phi_1)<v_2(\psi)$. Recall that $v_2(\phi_1\rightarrow\phi_2)=\min(v_1(\phi_1),v_2(\phi_2))$. Hence, $v_2(\phi_1\weakrightarrow\phi_2)<v_2(\psi)$, and $2:\phi_1\weakrightarrow\phi_2<2:\psi$ is satisfied as desired. By the same reasoning, we have that if $2:\phi_2\lesssim\mathfrak{X}$ is satisfied, then so is $2:\phi_1\weakrightarrow\phi_2\lesssim\mathfrak{X}$.

The cases of other rules can be showed in the same fashion.
\hfill $\Box$ \end{proof}
\subsection{Applications}\label{L2G2tableauxapplications}
\begin{corollary}\label{xytableauxcomplexity}
Satisfiability for any $\Luk^2_{(x,y)}(\rightarrow)$ and $\Luk^2_{(x,1)}(\weakrightarrow)$ is $\mathcal{NP}$-com\-plete while their validities are $co\mathcal{NP}$-complete.
\end{corollary}
\begin{proof}
Let $|\phi|$ be the number of symbols in $\phi$. Observe, from the proof of Theo\-rem~\ref{tableauxcompleteness}, that each tableau branch gives rise to two bounded mixed-integer programming problems (bMIP) --- each of the length $O(\rho(|\phi|))$ for some polynomial $\rho$. Recall that bMIP is $\mathcal{NP}$-complete (cf.~\cite{Haehnle1994}). Thus we can non-deterministically guess an open branch and then solve its two bMIPs (one arising from inequalities with $\leqslant_1$, and the other from those with $\leqslant_2$). 
This yields the $\mathcal{NP}$- and $co\mathcal{NP}$-membership for satisfiability and validity, respectively.

To obtain the $\mathcal{NP}$-hardness, we use the same method as in~\cite{Mundici1987,Haehnle1994}. For each classical formula $\phi$ one can construct a formula $\mathsf{two}(\phi)$ (cf. the detailed definition in~\cite[Lemmas~3.1--3.3]{Mundici1987}). Then by~\cite[Lemma~3.2]{Mundici1987}, $\phi$ is classically valid iff $\phi^{\mathsf{C}}\coloneqq\mathsf{two}(\phi)\supset\phi$ is \L-valid. 
Furthermore, if $\phi^{\mathsf{C}}$ is not valid, there is an \L-valuation $v$ such that $v(\phi^{\mathsf{C}})=0$. 
Recall that $\phi^\bullet$ (resp.\ $\phi^\circ$) denotes the formula obtained by substituting $\supset$ for $\rightarrow$ (resp.\ $\weakrightarrow$) in the formula $\phi$ (cf.\ Remark \ref{rk:conservative:extension}).
Thus, $\phi$ is classically valid iff $(\phi^{\mathsf{C}})^\bullet$ is $\Luk^2_{(1,0)}(\rightarrow)$-valid and $(\phi^{\mathsf{C}})^\circ$ is $\Luk^2_{(1,1)}(\weakrightarrow)$-valid. 
Furthermore, if $(\phi^{\mathsf{C}})^\bullet$ is not valid, there is $v'$ such that $v'((\phi^{\mathsf{C}})^\bullet)=(0,1)$, and $(\phi^{\mathsf{C}})^\circ$ is not valid, there is $v'$ such that  $v'((\phi^{\mathsf{C}})^\bullet)=(0,x)$ for some $x$. Since $(0,1)$ is not included in any non-trivial filter on $[0,1]\odot[0,1]$, and since no non-trivial filter can include $(1,1)$ and some $(0,x)$ simultaneously, we obtain
$$\vDash_{CPL}\phi\text{ iff }(\phi^{\mathsf{C}})^\bullet\text{ is }\Luk^2_{(x,y)}(\rightarrow)\text{-valid}\text{ iff }(\phi^{\mathsf{C}})^\circ\text{ is }\Luk^2_{(x,y)}(\weakrightarrow)\text{-valid}$$
as desired.
\hfill $\Box$ 
\end{proof}
\begin{remark}[Removing the branching]
We have introduced branching rules in our tableaux in order to make them more intuitive. It is possible, however, to make all rules \emph{linear} just as it was done originally in~\cite{Haehnle1994}. For example, the linear versions of $\rightarrow\leqslant_1$ and $\rightarrow\geqslant_2$ look as follows ($y\in\{0,1\}$):
\[\dfrac{\phi_1\rightarrow\phi_2\leqslant_1i}{\begin{matrix}\phi_1\geqslant_11-i+j-y&\quad y\leqslant i\\\phi_2\leqslant_1j+y&\quad j\leqslant i\end{matrix}}\quad\qquad
\dfrac{\phi_1\rightarrow\phi_2\geqslant_2i}{\begin{matrix}\phi_1\leqslant_2j+y&\quad y\leqslant1-i\\\phi_2\geqslant_2i+j-y&\quad j\leqslant1-i\end{matrix}}\]
Other rules can be easily acquired since $\vee$ and $\wedge$ can be defined via $\rightarrow$ and $\mathbf{0}$ in the language of $\Luk^2$. Rules without branching improve efficiency of the proof search by
removing the need to guess the branch whose bMIP we should solve.
\end{remark}
\begin{corollary}\label{Gtableauxcomplexity}
Satisfiability for $\mathsf{G}^2(\rightarrow)$ and $\mathsf{G}^2(\weakrightarrow)$ is $\mathcal{NP}$-complete.
\end{corollary}
\begin{proof}
It follows from the proof of theorem~\ref{Gconstraintcompleteness} that the satisfiability of $\mathsf{G}^2(\rightarrow)$ and $\mathsf{G}^2(\weakrightarrow)$ is in $\mathcal{NP}$: we obtain the valuation from~\eqref{valuationchain}, and it takes polynomial time to check that it indeed satisfies the formula.

The $\mathcal{NP}$-hardness follows since $\mathsf{G}^2$'s are conservative extensions of $\mathsf{G}$ whose satisfiability and validity are $\mathcal{NP}$- and $co\mathcal{NP}$-complete respectively. 
\hfill $\Box$ \end{proof}
We can also use the tableaux to check whether a set $\Gamma$ of assumptions entails a formula $\phi$ in the logics we consider. This yields the finite strong completeness for $\mathsf{G}^2$'s and $\Luk^2$'s by means of tableaux, and extends the complexity results to the finitary entailment.
\begin{corollary}
Let $\Gamma$ be a finite set of formulas. Then $\Gamma\vDash_{\Luk^2_{(x,y)}}\phi$ iff the left tableau closes, $\Gamma\vDash_{\mathsf{G}^2(\rightarrow)}\phi$ iff the central tableau closes, and $\Gamma\vDash_{\mathsf{G}^2(\weakrightarrow)}\phi$ iff the left tableau closes.
\begin{center}
\begin{forest}
smullyan tableaux
[{\bigcup\limits_{\phi'\in\Delta}\{\phi'\geqslant_1x,\phi'\leqslant_2y\}}[\phi\geqslant_1c[c<x]][\phi\geqslant_2d[d>y]]]
\end{forest}
\begin{forest}
smullyan tableaux
[{\{1:\phi'\geqslant{1},\mathfrak{\phi}'\leqslant{0}\mid\phi'\in\Gamma\}}[1:\phi<{1}][2:\phi>{0}]]
\end{forest}
\begin{forest}
smullyan tableaux
[{\{1:\phi'\geqslant{1}\mid\phi'\in\Gamma\}}[1:\phi<{1}]]
\end{forest}
\end{center}
Thus, the finitary entailment for any of these logics is $co\mathcal{NP}$-complete. 
\end{corollary}
%
\section{Conclusions and further research}
\label{sec:conclusion}

Using constraint tableaux, we have provided a modular treatment of the \L{}u\-ka\-sie\-wicz  and G\"{o}del based two-dimensional logics.
%
Our next steps are: (1) to study the structural proof theory of these logics, and of the two layer logics introduced in \cite{BilkovaFrittellaMajerNazari2020};
(2) to study and compare the logics in terms of consequence relations, to provide a Hilbert style axiomatization (for those where modus ponens is sound), and to prove standard completeness --- cases we understand so far are the following four: $\Luk _{(1,1)}^2(\weakrightarrow)$, $\Luk_{(1,0)}^2(\rightarrow)$ which is the logic $\Luk_{(\neg)}$ of \cite{BilkovaFrittellaMajerNazari2020}, $\mathsf{G}_{(1,0)}^2(\rightarrow)$ whose validities coincide with the axiomatic extension of Wansing's $I_4C_4$ \cite{Wansing2008} with the prelinearity axiom, and $\mathsf{G}_{(1,1)}^2(\weakrightarrow)$ whose consequence coincides with the axiomatic extension of Nelson's $N4^\bot$ \cite{Nelson1949,Odintsov2008} with the prelinearity axiom. 

In a broader sense we naturally aim to provide a general treatment of two-dimensional graded logics. Indeed, within the research project introduced in \cite{BilkovaFrittellaMajerNazari2020}, we want to develop a modular logical framework for reasoning based on heterogeneous information (such as crisp or fuzzy  data, personal beliefs, etc.) that can be both incomplete and inconsistent. In addition, we do not wish to commit to a specific logic to model the reasoning of the agent(s), because different situations may call for different logics --- modeling the reasoning of a~group of experts is different from modeling the reasoning of the crowd. Doing so requires the ability to manipulate and combine logics for these different situations in a modular way.
\bibliographystyle{splncs04}

\section{Appendix}
\label{sec:appendix}



\subsection{Proofs of Section 
\ref{ssec:sem:properties:L2}. Semantical properties of $\Luk^2_{(x,y)}(\rightarrow)$}
\label{appendix:proof:L2}

\subsubsection{Proposition \ref{prop:rectangles:to:squares}}
\begin{itemize}[noitemsep,topsep=2pt]
\item Let $y\geq 1-x$. Then $\phi$ is $\Luk^2_{(x,y)}(\rightarrow)$-valid iff $\phi$ is $\Luk^2_{(x,1-x)}(\rightarrow)$-valid.
\item Let $y<1-x$. Then $\phi$ is $\Luk^2_{(x,y)}(\rightarrow)$-valid iff $\phi$ is $\Luk^2_{(1-y,y)}(\rightarrow)$-valid.
\end{itemize}

\begin{proof}
We prove only the first case, as the second one can be tackled in the same manner. Assume that $\phi$ is $\Luk^2_{(x,1-x)}(\rightarrow)$-valid. Since $y\geq 1-x$, $(x,1-x)^\uparrow\subseteq(x,y)^\uparrow$, then $\phi$ is $\Luk^2_{(x,y)}(\rightarrow)$-valid as well.
Now let us show that if $\phi$ is $\Luk^2_{(x,y)}(\rightarrow)$-valid, then $\phi$ is $\Luk^2_{(x,1-x)}(\rightarrow)$-valid as well.

For any valuation $v$, let us define the valuation $v^*$ as follows: $v^*(p)\coloneqq(1-v_2(p),1-v_1(p))$. We now show that the extension of $v^*$ satisfies the following property: for every formula $\phi$, we have $v^*(\phi)=(1-v_2(\phi),1-v_1(\phi))$.

Notice that $v(\phi)=(x',y')\Leftrightarrow v(\neg{\sim}\phi)=(1-y',1-x')$ 
and $v(\psi_1\leftrightarrow\psi_2)=(1,0)\Leftrightarrow v(\psi_1)=v(\psi_2)$. Since $(1,0)\in(x,y)^\uparrow$ for any $x$ and $y$,  a formula  always evaluated at $(1,0)$ is necessarily $\Luk^2_{(x,y)}(\rightarrow)$-valid. Notice that
\[\begin{array}{rcl}
v(\neg{\sim}\neg\psi_1\leftrightarrow\neg\neg{\sim}\psi_1)&=&(1,0),
\\
v(\neg{\sim}(\psi_1\wedge\psi_2)\leftrightarrow(\neg{\sim}\psi_1\wedge\neg{\sim}\psi_2))&=&(1,0),\\
v(\neg{\sim}(\psi_1\vee\phi_2)\leftrightarrow(\neg{\sim}\psi_1\vee\neg{\sim}\psi_2))&=&(1,0),\\
v(\neg{\sim}(\psi_1\rightarrow\phi_2)\leftrightarrow(\neg{\sim}\psi_1\rightarrow\neg{\sim}\psi_2))&=&(1,0).
\end{array}\]
Indeed, it follows from the fact that these formulas are derivable in the Hilbert calculus for $\Luk_{(\neg)}$ from~\cite[Lemma~2]{BilkovaFrittellaMajerNazari2020} which is sound and complete w.r.t. $[0,1]_{\Luk}\odot[0,1]_{\Luk}(\rightarrow)_{(1,0)}$. Using the four equalities above, one can prove by induction on $\phi$ that $v^* (\phi) = (1-v_2(\phi), 1-v_1(\phi))$ for any formula $\phi$.
Thus for every valuation $v$ we have defined its counterpart $v^*$ such that $v^*(\phi)$ is the reflection of $v(\phi)$ along the vertical axis of the lattice. 

Let $\phi$ be $\Luk^2_{(x,y)}(\rightarrow)$-valid but not $\Luk^2_{(x,1-x)}(\rightarrow)$-valid. Then, there is a $v$ s.t.\ $v(\phi)\in(x,y)^\uparrow$ but $v(\phi)\notin(x,1-x)^\uparrow$. But then $v^*(\phi)\notin(x,y)^\uparrow$ since only $(x,1-x)^\uparrow$'s are closed under conflation. Contradiction.
\hfill $\Box$ \end{proof}

\subsubsection{Proposition \ref{smallsquares}}
Let $m,n\in\{2,3,\ldots\}$. Then, $\Luk^2_{\left(\frac{m-1}{m},\frac{1}{m}\right)}\subsetneq\Luk^2_{\left(\frac{n-1}{n},\frac{1}{n}\right)}$ iff $m>n$.

\begin{proof}
The inclusion follows immediately from the fact that $$\left(\dfrac{m-1}{m},\dfrac{1}{m}\right)^\uparrow\subsetneq\left(\dfrac{n-1}{n},\dfrac{1}{n}\right)^\uparrow\text{ iff }m>n$$
To show the strictness of the inclusion, consider the following family of formulas.
\[\mathsf{F}_n\coloneqq\bigvee\limits_{\begin{matrix}1\leq i<j\leq n+1\end{matrix}}\left(p_i\leftrightarrow p_j\right)\]
Notice that $v(\phi_1\leftrightarrow\phi_2)=(1-|v_1(\phi_1)-v_1(\phi_2)|,|v_2(\phi_1)-v_2(\phi_2)|)$. In addition, we have that for $n\geq 2$:
\begin{equation}
\label{eq:Fn:geq}
  v(\mathsf{F}_n)\geq \left(\dfrac{n-1}{n},\dfrac{1}{n}\right)  
\end{equation}
Indeed, we can observe that $v_1(\psi_1\leftrightarrow\psi_2)$ is the \emph{complement to the distance} between $v_1(\psi_1)$ and $v_1(\psi_2)$ on $[0,1]$, while $v_2(\psi_1\leftrightarrow\psi_2)$ is the distance between $v_2(\psi_1)$ and $v_2(\psi_2)$. Thus, $v_1(\mathsf{F}_n)$ is the \emph{maximal complement to the distance between any two points out of $n$ on} $[0,1]$, while $v_2(\mathsf{F}_n)$ is the \emph{minimal such distance}. The lower bound on $v(\mathsf{F}_n)$ is produced when we place points on $[0,1]$ with \emph{equal intervals} between them.
Hence, in each $\Luk^2_{\left(\frac{k-1}{k},\frac{1}{k}\right)}(\rightarrow)$, only $\mathsf{F}_n$'s with $n\geq  k$
are valid. Furthermore, none of the $\mathsf{F}_n$'s are valid in $\Luk_{(1,0)}$. This gives us the desired strictness of the inclusion.
\hfill $\Box$ \end{proof}
\subsubsection{Proposition \ref{bigsquares}}
Let $m,n\in\{3,4,\ldots\}$. Then $\Luk^2_{\left(\frac{m-2}{2m},\frac{m+2}{2m}\right)}\!\subsetneq\!\Luk^2_{\left(\frac{n-2}{2n},\frac{n+2}{2n}\right)}$ iff $m\!>\!n$.

\begin{proof}
As in Proposition~\ref{smallsquares}, inclusion follows from the fact that
$$\left(\frac{m-2}{2m},\frac{m+2}{2m}\right)^\uparrow\subsetneq\left(\frac{n-2}{2n},\frac{n+2}{2n}\right)^\uparrow\text{ iff }m>n$$

Again, the non-trivial part is the strictness. Recall that $\psi_1\odot\psi_2={\sim}(\psi_1\rightarrow{\sim}\psi_2)$.
Now, consider the family of formulas $\mathsf{F}_2\odot\mathsf{F}_n$ for $n\geq 3$. 
We have that 
\begin{align*}
v(\mathsf{F}_2\odot\mathsf{F}_n)&=\left(\max\left(0,v_1(\mathsf{F}_2)+v_1(\mathsf{F}_n)-1\right),\min\left(1,v_2(\mathsf{F}_2)+v_2(\mathsf{F}_n)\right)\right)\\
&\geq \left(\max\left(0,\dfrac{1}{2}+\dfrac{n-1}{n}-1\right),\min\left(1,\dfrac{1}{2}+\dfrac{1}{n}\right)\right)
\tag{using \eqref{eq:Fn:geq}}
\\
&\geq \left(\dfrac{n-2}{2n},\dfrac{n+2}{2n}\right).
\end{align*}

As in Proposition~\ref{smallsquares}, we can see that $\mathsf{F}_2\odot\mathsf{F}_k$ is $\Luk_{\left(\frac{n-2}{2n},\frac{n+2}{2n}\right)}$-valid iff $k\geq  n$.
\hfill $\Box$ \end{proof}

\subsubsection{Proposition \ref{nomodusponens}}
Let $\Luk^2_{(1,0)}(\rightarrow)\subsetneq\Luk^2_{(x,y)}(\rightarrow)$. Then $\Luk^2_{(x,y)}(\rightarrow)$ is not closed under modus ponens.
\begin{proof}
Observe that by proposition~\ref{prop:rectangles:to:squares}, $\Luk^2_{(1,0)}(\rightarrow)\subsetneq\Luk^2_{(x,y)}(\rightarrow)$ iff $x<1$ and $y>0$. Furthermore, by proposition~\ref{prop:rectangles:to:squares}, it suffices to consider only the $\Luk^2_{(x',1-x')}(\rightarrow)$'s. Since for any $\psi$ and $\psi'$, $\psi\rightarrow(\psi'\rightarrow(\psi\odot\psi'))$ is $\Luk^2_{(1,0)}$-valid, it is enough to find $\Luk^2_{(x,1-x)}(\rightarrow)$-valid formulas $\psi$ and $\psi'$ such that $\psi\odot\psi'$ is not $\Luk^2_{(x',1-x')}(\rightarrow)$-valid.

We consider two cases. Either (1) $(x',1-x')^\uparrow$ is not an extension of $\left(\frac{1}{2},\frac{1}{2}\right)^\uparrow$ or (2) it is.
In the first case, we have $\frac{1}{2} < x' \leq 1$. Hence, there is a $k\in\mathbb{N}$ s.t.
$k \geq 3$ and $\frac{k-2}{k-1}<x'\leq\frac{k-1}{k}$. Hence, $\mathsf{F}_{k}$ is $\Luk^2_{(x',1-x')}(\rightarrow)$-valid.
By Proposition \ref{smallsquares}, there is $v$ such that $v(\mathsf{F}_{k-1}) = (\frac{k-2}{k-1},\frac{1}{k-1}) \notin (x',1-x')^\uparrow$.
Hence, $\mathsf{F}_{k-1}$ is not $\Luk^2_{(x',1-x')}(\rightarrow)$-valid. 
Observe that in such case $\mathsf{F}_{k}\odot\mathsf{F}_{k}$ is not $\Luk^2_{(x',1-x')}(\rightarrow)$-valid. Indeed, consider the valuation $v'$ such that $v'(p_1)=(1,0)$, $v'(p_{k+1})=(0,1)$, and $v'(p_i)=\frac{i-1}{k}$ for every $2\leq i\leq k$. We get that
\begin{align*}
  v'(\mathsf{F}_k\odot\mathsf{F}_k)=\left(\frac{k-2}{k},\frac{2}{k}\right) < (x', 1-x').  
  \tag{because $\frac{k-2}{k-1}< x'$}
\end{align*}
Hence, $\mathsf{F}_k\odot\mathsf{F}_k$ is not $\Luk^2_{(x',1-x')}(\rightarrow)$-valid as desired.


For the second case, we have $0 \leq x \leq \frac{1}{2}$.
Note  that $v'(\mathsf{F}_2\odot\mathsf{F}_2)=(0,1)$ and thus not valid in any $\L^2_{(x,y)}$. This proves the second case as $\mathsf{F}_2$ is $\Luk^2_{(x,y)}$-valid iff $(x,y)^\uparrow\supseteq(\frac{1}{2},\frac{1}{2})^\uparrow$.
\hfill $\Box$ \end{proof}

\subsection{Proofs of Section \ref{ssec:sem:properties:G2}. Semantical properties of $\mathsf{G}^2(\rightarrow)$}
\label{appendix:proof:G2}

\subsubsection{Proposition \ref{no1no0}}
Let $\phi$ be a formula over $\{\mathbf{0},\mathbf{1},\neg,\wedge,\vee,\rightarrow,\Yleft\}$. For any $v(p)=(x,y)$, let $v^*(p)=(1-y,1-x)$. Then $v(\phi)=(x,y)$ iff $v^*(\phi)=(1-y,1-x)$.

\begin{proof}
Let $\chi$ and $\chi'$ be formulas. Let us show that,
for all $\mathbf{x},\mathbf{x'},\mathbf{y},\mathbf{y'}\in\{1,2\}$ such that $\mathbf{x}\neq\mathbf{y}$ and $\mathbf{x'}\neq\mathbf{y'}$, we have
\begin{equation}\tag{$*$}\label{dualorder}
v_\mathbf{x}(\chi)\geq  v_\mathbf{x'}(\chi')\Leftrightarrow v^*_\mathbf{y}(\chi')\geq  v^*_\mathbf{y'}(\chi)
\end{equation}
We proceed by the induction on the number of unary and binary connectives in both $\chi$ and $\chi'$.
The only non-trivial case is that of $\rightarrow$. For $\rightarrow$, we have
\begin{align*}
v_1(\psi_1\rightarrow\psi_2)<v_\mathbf{x}(\chi) \quad
\Leftrightarrow\quad&v_1(\psi_1)>v_1(\psi_2) \text{ and } v_1(\psi_2)<v_\mathbf{x}(\chi)
\\
\Leftrightarrow\quad&v^*_2(\psi_1)<v^*_2(\psi_2)
 \text{ and }
v^*_2(\psi_2)>v^*_\mathbf{y}(\chi)\tag{by IH}\\
\Leftrightarrow\quad&v^*_2(\psi_1\rightarrow\psi_2)>v^*_\mathbf{y}(\chi)
\end{align*}
and
\begin{align*}
& v_1(\psi_1\!\rightarrow\!\psi_2)\!=\!v_\mathbf{x}(\chi)
\\
\Leftrightarrow \quad & v_1(\psi_1)\!>\!v_1(\psi_2)\!=\!v_\mathbf{x}(\chi)\text{ or } (\,  v_1(\psi_1)\!\leq\!v_1(\psi_2) \text{ and } v_\mathbf{x}(\chi)\!=\!v_1(\mathbf{1}) \, )
\\
\Leftrightarrow \quad&
v^*_2(\psi_1)\!<\!v^*_2(\psi_2)\!=\!v^*_\mathbf{y}(\chi)\text{ or } (\, v^*_2(\psi_1)\!\geq \!v^*_2(\psi_2) \text{ and } v^*_\mathbf{y}(\chi)\!=\!v^*_2(\mathbf{1}) \,)
\tag{by IH}
\\
\Leftrightarrow \quad
&v^*_2(\psi_1\rightarrow\psi_2)=v^*_\mathbf{y}(\chi)
\end{align*}
\begin{comment}
\blue{
and also
\begin{align*}
v_2(\psi_1\rightarrow\psi_2)\leq v_\mathbf{x}(\chi)&\Leftrightarrow v_2(\psi_2)\leq v_2(\psi_1)\text{ or }v_2(\psi_2)\leq v_\mathbf{x}(\chi)\\
&\Leftrightarrow v_1(\psi_2)\geq  v_1(\psi_1)\text{ or }v_1(\psi_2)\geq  v_\mathbf{y}(\chi)\tag{by IH}\\
&\Leftrightarrow v_1(\psi_1\rightarrow\psi_2)\geq  v_\mathbf{y}(\chi)
\end{align*}
}
\end{comment}

Now we can prove the statement by  induction on $\phi$. The basis cases of variables and constants hold by the construction of $v^*$.
\begin{comment}
\textcolor{blue}{For $\neg$, we have
\begin{align*}
v(\neg\psi)=(x,y)&\Leftrightarrow v(\psi)=(y,x)\\
&\Leftrightarrow v^*(\psi)=(1-x,1-y)\tag{by IH}\\
&\Leftrightarrow v^*(\neg\psi)=(1-y,1-x)
\end{align*}
}

\textcolor{blue}{For $\wedge$, we assume w.l.o.g. that $v_1(\psi_1)\leq v_1(\psi_2)$ and $v_2(\psi_1)\leq v_2(\psi_2)$. Furthermore, we let $v(\psi_1)=(x_1,y_1)$ and $v(\psi_2)=(x_2,y_2)$ we have
\begin{align*}
v(\psi_1\wedge\psi_2)=(x_1,y_2)&\Leftrightarrow v(\psi_1)=(x_1,y_1)\text{ and }v(\psi_2)=(x_2,y_2)\tag{by ass.}\\
&\Leftrightarrow v^*_1(\psi_1)=(1-y_1,1-x_1),v^*_2(\psi_2)=(1-y_2,1-x_2)\tag{by IH}\\
&\Leftrightarrow v^*(\psi_1\wedge\psi_2)=(1-y_2,1-x_1)\tag{by ass.}
\end{align*}
}
\end{comment}
We only present the case of $\rightarrow$.
We consider two cases: $(x,y)\neq(1,0)$ and $(x,y)=(1,0)$.

In the first and in the second cases, we have
\begin{align*}
&\quad \ \ v(\psi_1\rightarrow\psi_2)=(x,y) \\ 
& \Leftrightarrow  v_1(\psi_1)>v_1(\psi_2)=x\text{ and }v_2(\psi_1)<v_2(\psi_2)=y\\
& \Leftrightarrow  v^*_2(\psi_1)<v^*_2(\psi_2)=1-x\text{ and }v^*_1(\psi_1)>v^*_1(\psi_2)=1-y \tag{by IH and~\eqref{dualorder}}\\
& \Leftrightarrow  v^*(\psi_1\rightarrow\psi_2)=(1-y,1-x)
\end{align*}
and
\begin{align*}
v(\psi_1\rightarrow\psi_2)=(1,0)&\Leftrightarrow v_1(\psi_1)\leq v_1(\psi_2)\text{ and }v_2(\psi_1)\geq  v_2(\psi_2)\\
&\Leftrightarrow v^*_2(\psi_1)\geq  v^*_2(\psi_2)\text{ and }v^*_1(\psi_1)\leq v^*_1(\psi_2)
\tag{by IH and~\eqref{dualorder}}\\
&\Leftrightarrow v^*(\psi_1\rightarrow\psi_2)=(1,0).
\end{align*}
\hfill $\Box$ \end{proof}
In order to prove Proposition \ref{prop:no:lower:limits:for:G:neg}, we need the following claims.

\paragraph{Claim 1. (NNF)}
\label{lem:nnf}
Each formula over $\{\mathbf{0},\mathbf{1},\neg,\wedge,\vee,\rightarrow,\Yleft\}$ is equivalent to a formula in a $\neg$-negation normal form (NNF) in $\mathsf{G}^2(\rightarrow)$.

\begin{proof}
We create the NNF by repeating applications of instances of the following formulas, which are always designated (as they are always evaluated at $(1,0)$). 
\begin{align*}
\neg\neg\phi &\leftrightarrow \phi & \neg(\phi\wedge\psi) &\leftrightarrow (\neg\phi\vee\neg\psi)\\
\neg \mathbf{0} &\leftrightarrow \mathbf{1} & \neg(\phi\vee\psi) &\leftrightarrow (\neg\phi\wedge\neg\psi)\\
\neg \mathbf{1} &\leftrightarrow \mathbf{0} & \neg(\phi\rightarrow\psi) &\leftrightarrow (\neg\psi\Yleft\neg\phi)\\
 & & \neg(\phi\Yleft\psi) &\leftrightarrow (\neg\psi\rightarrow\neg\phi)
\end{align*}
A similar claim was shown in \cite{BilkovaFrittellaMajerNazari2020} for the logic $\Luk_{(\neg)}$ which coincides with $\Luk^2_{(1,0)}(\rightarrow)$, and can be shown for all the logics we consider in this paper, using the corresponding implication and varying the way we negate it (the one but last formula on the right). For logics with $ \weakrightarrow$ we only get weakly equivalent negation normal form.
\hfill $\Box$
\end{proof}

\paragraph{Claim 2.} For every G\"odel formula $\phi$ over $\{\mathbf{0},\mathbf{1},\wedge,\vee,\rightarrow,\Yleft\}$, for every valuation $v$ such that $v(\phi) < 1$ and for every $0 < x \leq 1 $, we have that
\begin{enumerate}[noitemsep,topsep=2pt]
    \item $v(\phi) \leq \max \{ v(p) \mid p \in \mathsf{Var}(\phi) \}$ and
    \item there exists a valuation $v'$ such that $v'(\phi) \leq x $.
\end{enumerate}

\begin{proof}
1) We proceed by induction. The only non-trivial parts are $\phi=\phi_1\rightarrow\phi_2$ and $\phi=\phi_1\Yleft\phi_2$.

If $v(\phi_1\rightarrow\phi_2)<1$, then $v(\phi_2)<v(\phi_1)$ and $v(\phi_1\rightarrow\phi_2)=v(\phi_2) < 1$.
Let $q \in \mathsf{Var}(\phi)$ be such that $v(q)=\max\{v(p)\mid p\in\mathsf{Var}(\phi)\}$. We consider two cases: either $q\in\mathsf{Var}(\phi_1)$ or $q\in\mathsf{Var}(\phi_2)$. 
In both cases we obtain the desired result by the induction hypothesis.

If $v(\phi_1\Yleft\phi_2)<1$, then $v(\phi_1)<1$ and $v(\phi_1\Yleft\phi_2)=v(\phi_1)$ or $v(\phi_1\Yleft\phi_2)=0$ and $v(\phi_1)\leq v(\phi_2)$. 
Let $q \in \mathsf{Var}(\phi)$ be such that  $v(q)=\max\{v(p)\mid p\in\mathsf{Var}(\phi)\}$. 
We consider two cases: either $q\in\mathsf{Var}(\phi_1)$ or $q\in\mathsf{Var}(\phi_2)$. In both cases we obtain the desired result by the induction hypothesis.

2) Let $v(\phi)<1$, $n > 0$ and $\frac{1}{n} \leq x$. 
We construct $v'(\phi)$ as follows: $v'(p)=\frac{v(p)}{n}$. We can prove by induction on $\phi$ that $v'(\phi) < 1$. Hence, by the previous item, we have that 
$$v'(\phi) \leq \max \{ v'(p) \mid p \in \mathsf{Var}(\phi) \} = \frac{1}{n} \max \{ v(p) \mid p \in \mathsf{Var}(\phi) \} \leq x$$ 
as required.
\hfill $\Box$ \end{proof}

\paragraph{Claim 3.}
Let $\phi$ be a formula over $\{\mathbf{0},\mathbf{1},\neg,\wedge,\vee,\rightarrow,\Yleft\}$, for every valuation $v$ such that $v_1(\phi) \neq 1$ and for every $0 < x \leq 1 $, we have that
\begin{enumerate}[noitemsep,topsep=2pt]
    \item $v_1(\phi) \leq \max \left( \{v_1(p)\mid p\in\mathsf{Var}(\phi)\} \cup \{v_2(p)\mid p\in\mathsf{Var}(\phi)\} \right)$ and
    \item there exists a valuation $v'$ such that $v_1'(\phi) \leq x $.
\end{enumerate}

\begin{proof}
W.l.o.g.\ we can assume that $\phi$ is in NNF. 

1) Let $v$ be a valuation such that $v_1(\phi) \neq 1$.
$\phi$ being a formula in NNF in the language $\{\mathbf{0},\mathbf{1},\neg,\wedge,\vee,\rightarrow,\Yleft\}$ over the set $\mathsf{Prop}$, it can be perceived as a formula $\phi'$ in the language  $\{\mathbf{0},\mathbf{1},\wedge,\vee,\rightarrow,\Yleft\}$ over the set of literals $\mathsf{Lit}$ in place of propositional atoms. 
Let $w$ be the new valuation over the set of literals defined as follows: $w(p) = v_1(p)$ and $w(\neg p)=v_1(\neg p)=v_2(p)$.
By applying Claim 2 to $\phi'$ and $w$, we get that 
\begin{align*}
    w(\phi') & \leq \max \{ w(l) \mid l \in \mathsf{Lit}(\phi') \}\\ 
    & = \max \left( \{v_1(p)\mid p\in\mathsf{Var}(\phi)\} \cup \{v_2(p)\mid p\in\mathsf{Var}(\phi)\} \right)
\end{align*}
as required.

2) By Claim 2, there exists $w'$ on the set of literals such that $w'(\phi') \leq x$. Let $v'$ be defined as follows:
$v'(p) = (w(p), w(\neg p))$ for every $p \in \mathsf{Prop}$.
We get that $v'(\phi) \leq x$
 as required.
\hfill $\Box$  \end{proof}

\subsubsection{Proposition \ref{prop:no:lower:limits:for:G:neg}}
Let $\phi$ be a formula over $\{\mathbf{0},\mathbf{1},\neg,\wedge,\vee,\rightarrow,\Yleft\}$  such that $v(\phi)\geq (x,y)$ for any $v$ and some fixed $(x,y)\neq(0,1)$. Then $v'(\phi)=(1,0)$ for any $v'$.

\begin{proof}
Assume that $\phi$ is $\mathsf{G}^2_{(x,y)}$-valid and w.l.o.g. that $\phi$ is in NNF.
Observe that $(x,x)$ points are not affected by $\neg$. 
Now recall that by proposition~\ref{no1no0}, if $v(\phi)\neq(1,0)$, there is a $v'$ such that $v'_1(\phi)\neq1$. 
Furthermore, notice that if $\{(0,0),(1,1)\}\subseteq(x,y)^\uparrow$ then $(x,y)^\uparrow=(0,1)$. 
Hence, we have that $$\{(0,0),(1,1)\}\not\subseteq(x,y)^\uparrow,$$ which implies that
\[\begin{array}{c}
\exists(z,z)\forall(x',y')\in(x,y)^\uparrow:(z,z)\neq(1,1)\text{ and }z\geq  y'\\
\text{or}\\
\exists(z,z)\forall(x',y')\in(x,y)^\uparrow:(z,z)\neq(0,0)\text{ and }z\leq x'.\\
\end{array}\]
By proposition~\ref{no1no0}, 
we know that $v(\phi)=(0,0)$ iff $v^*(\phi)=(1,1)$, hence
we can state w.l.o.g. that $(0,0)\notin(x,y)^\uparrow$ and that $v(\phi) \neq (0,0)$ for every $v$. 
Thus, there is a~$(z,z)$ such that $(z,z)\neq(0,0)$ and for any $(x',y')\in(x,y)^\uparrow$ $z\leq x'$.

Assume, for contradiction, that
$v'(\phi)\neq(1,0)$. There are two cases.

Case 1: $v'_1(\phi)\neq 1$. 
Since $x>0$, by Claim 3, we have that 
there exists a valuation $v'$ such that $v_1'(\phi) \leq \frac{x}{2} < x$. Hence,  $v'(\psi)\notin(x,y)^\uparrow$, which contradicts the fact that $\psi$ is $\mathsf{G}^2_{(x,y)}$-valid.

Case 2: $v'_1(\phi)=1$. Then $v'_2(\phi)\neq 0$ and $v'^*_1(\phi) = 1 - v'_2(\phi) \neq 1$ (see Proposition \ref{no1no0} for definition of $v'^*$). We proceed as in the previous case.

\hfill $\Box$ \end{proof}

\subsubsection{Corollary \ref{cor:not1=not0}}
$v(\phi)=(1,0)$ for any $v$ iff $v'_1(\phi)=1$ for any $v'$.

\begin{proof}
Left to right is obvious. We prove right to left by contraposition.
Assume that $v(\phi)\neq(1,0)$. Then, either $v_1(\phi)\neq1$ or $v_2(\phi)=c\neq0$. In the first case, we get the result. In the second case we obtain by proposition~\ref{no1no0} that $1\neq v^*_1(\phi)=1-c$ as required.
\hfill $\Box$ \end{proof}

\end{document}